\documentclass{amsart}
\pdfoutput=1

\usepackage{graphicx}
\usepackage{mathptmx}    
\usepackage{latexsym}
\usepackage{amsfonts}
\usepackage{amsmath}

\newtheorem{theorem}{Theorem}[section]
\newtheorem{lemma}[theorem]{Lemma}
\newtheorem{corollary}[theorem]{Corollary}
\newtheorem{proposition}[theorem]{Proposition}

\theoremstyle{definition}

\theoremstyle{remark}

\numberwithin{equation}{section}

\begin{document}
\title{A triangulation of semi-algebraic sets concerning an analytical condition for shortest-length curves}
\author{Chengcheng Yang }
\address{Department of Mathematics, Rice University, 6100 Main Street, Houston TX 77005}
\email{cy2@rice.edu}           
\thanks{The author was support in part by NSF Award\#1745670}
\date{\today}

\begin{abstract}
This paper concerns an analytical stratification question of real algebraic and semi-algebraic sets. In 1957 Whitney \cite{W} gave a stratification of real algebraic sets, it partitions a real algebraic set into partial algebraic manifolds. In 1975 Hironaka \cite{H} reproved that a real algebraic set is triangulable and also generalized it to semi-algebraic sets, following the idea of Lojasiewicz's \cite{L} triangulation of semi-analytic sets in 1964. Following their examples and wondering how geometry looks like locally. this paper tries to come up with a stratification, in particular a cell decomposition, such that it satisfies the following analytical property. Given any shortest curve between two points in a real algebraic or semi-algebraic set, it interacts with each cell at most finitely many times. 
\end{abstract}

\maketitle

\section{Introduction}
This paper concerns an analytical stratification question of real algebraic and semi-algebraic sets. A real algebraic set defined by polynomials $f_1, \ldots, f_m$ is 
\begin{equation}
\{x \in \mathbb{R}^n: \, f_i(x) = 0 \text{ for all } i = 1, \ldots, m\}. \nonumber \end{equation}
The triangulability question for algebraic sets was first considered by van de Waerden \cite{V} in 1929. It is a well-known theorem that every algebraic set is triangulizable \cite{H}. On the other hand, in 1957 Whitney \cite{W} introduced another splitting process that divides a real algebraic $V$ into a finite union of ``partial algebraic manifolds." An algebraic partial manifold $M$ is a point set, associated with a number $\rho$, with the following property. Take any $p \in M$. Then there exists a set of polynomials $f_1, \ldots, f_{\rho}$, of rank $\rho$ at $p$, and a neighborhood $U$ of $p$, such that $M \cap U$ is the set of zeros in $U$ of these $f_i$. The splitting process uses the rank of a set $S$ of functions $f_1, \ldots, f_s$ at a point $p$, where the rank of $S$ at $p$ is the number of linearly independent differentials $df_1(p), \ldots, df_s(p)$. 

A basic (closed) semi-algebraic set defined by polynomials $f_1, \ldots, f_m$ is 
\begin{equation}
\{x \in \mathbb{R}^n: \, f_i(x) \geq 0 \text{ for all } i = 1, \ldots, m\}. \nonumber \end{equation}
Given finitely many basic semi-algebraic sets, we may take their unions, intersections, and complements. In general, a semi-algebraic set may be represented as a finite union of sets of the form 
\begin{equation}
\{x \in \mathbb{R}^n: \, f_i(x) > 0, h_j = 0, \text{ for all } i = 1, \ldots, m, j= 1, \ldots, p\}. \nonumber
\end{equation}
In 1975 Hironaka \cite{H} proved that every semi-algebraic set is also triangulable. His proof came from a paper of Lojasiewicz \cite{L} in 1964, in which he proved that a semi-analytic set admits a semi-analytic triangulation. 

%This paper concerns the stratification question whether Whitney's stratification or Lojasiewicz/Hironaka's triangulation admits the analytical property that any shortest curve between two point in the set interact with each building block (whether an algebraic partial manifold or a cell complex) finitely many times. 

Following the examples of Whitney's stratification and Lojasiewicz/Hironaka's triangulation, this paper tries to come up with a cell-complex stratification that admits the analytical condition that any shortest curve between two points in a real algebraic or semi-algebraic set interacts with each cell finitely many times. 

We begin with investigating the real algebraic or semi-algebraic sets in $\mathbb{R}^2$. 

\section{Region below the graph of a polynomial function in $\mathbb{R}^2$}
\label{sec:2} 

We begin with a closed region that is below the graph of a polynomial function. A general picture looks like the one as shown in Figure~\ref{fig:1}. We want to study the behavior of a shortest-length curve between any two points in the region. In particular, we look for a cell decomposition of the region so that a shortest-length curve `interacts' with each open cell at most finitely many times.  When we say that a curve {\bf interacts with a cell once}, we mean that the curve enters and exits a cell once. {\bf A shortest-length curve} between any two points is assumed to be $C^2$-path with the minimum length among all possible $C^2$-curves connecting the two points. We know that if such a curve exists, its length can be approximated by the length of a polygonal line inscribed in the curve as close as possible (see [1]). To be more flexible, we allow {\it piecewise} $C^2$-curves to be included in our study for curves of shortest lengths. 

\begin{figure}[ht]
\includegraphics[width=5cm]
{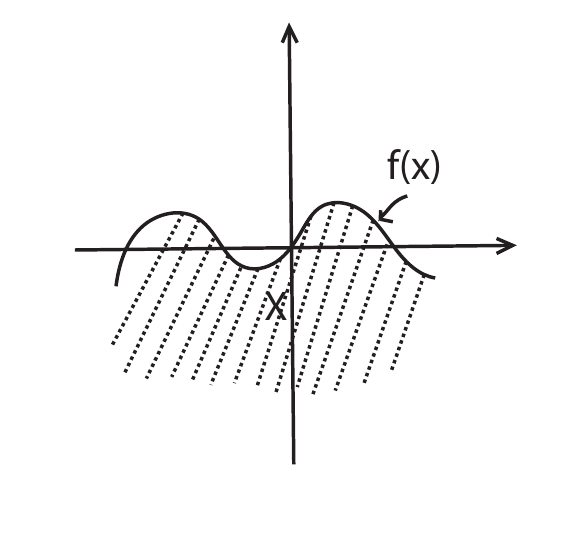} 
\caption{A region below the graph of $f(x)$ including the boundary.}
\label{fig:1}
\end{figure}

Suppose $f(x)$ is a polynomial function, $X$ is the closed region below the graph of $f(x)$, and $A, B$ are two arbitrary points inside $X$. Without loss of generality, we assume that $A$ is to the left of $B$. Assume $\gamma: [a, b] \rightarrow \mathbb{R}^2$ is a piecewise $C^2$-curve from $A$ to $B$ that is contained in $X$ and has the shortest length among all possible piecewise $C^2$-curves from $A$ to $B$ in $X$. We call such a curve {\bf a shortest-length curve from $A$ to $B$ in $X$}. We want to come up with a cell decomposition of $X$ such that $\gamma$ interacts with each cell at most finitely many times.

If the degree of $f$ is less than or equal to 1, $f$ is a straight line and $X$ is convex, so $\gamma$ is the straight line segment from $A$ to $B$. A cell decomposition of $X$ is described as follows. Since $X$ is a half plane, we can divide $X$ into grids as shown in Figure~\ref{fig:15}. The 0- and 1-cells on the graph of $f(x)$ are $(n, f(n))$, for $n \in \mathbb{Z}$, and the open intervals between them. The rest are parallel transports of them. The 2-cells are the squares.
We can show that $\gamma$ interacts with each cell at most once. Suppose not, we can pick a point from one interaction and another point from a second interaction, then the shortest path between these two points is either a straight line segment or a constant path. Since each cell is convex, $\gamma$ should stay in one cell, which is a contradiction. 

\begin{figure}[ht]
\includegraphics[width=5cm]
{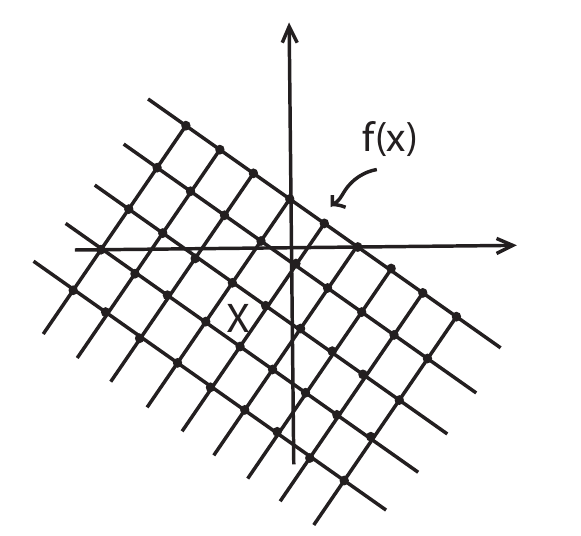} 
\caption{A cell decomposition of $X$ when $deg(f) \leq 1$.}
\label{fig:15}
\end{figure}

If the degree of $f$ is greater than 1, $f''$ is not identically zero, so both $f''$ and $f'$ have only finitely many zeros. We can check those zeros one by one to find out all the strict inflection points and the local minimum points of $f$. Here, a strict inflection point means that $f''$ changes signs from one side to the other. Then we consider the cell decomposition of $X$ as shown in Figure~\ref{fig:16}.

The idea is similar as before. Here we need to add a few more 0-cells on the graph of $f(x)$ and the picture looks more like a brick pattern. Let the 0-cells on the graph be $(n, f(n))$, $n \in \mathbb{Z}$, plus the strict inflection points and the local minimum points. For every pair of consecutive 0-cells on the graph, draw two vertical rays below those two points. Since the graph has a minimum point between the two 0-cells, we put a horizontal 1-cell at a distance of 1 unit below the minimum point, then put a second 1-cell at a distance of 1 unit below the first one, and continue in this way down. Repeat this process for all the other pairs of consecutive 0-cells on the graph, so we divide $X$ into grids thus get a cell decomposition of $X$. We contend that $\gamma$ interacts each cell at most finitely many times.

\begin{figure}[ht]
\includegraphics[width=8cm]
{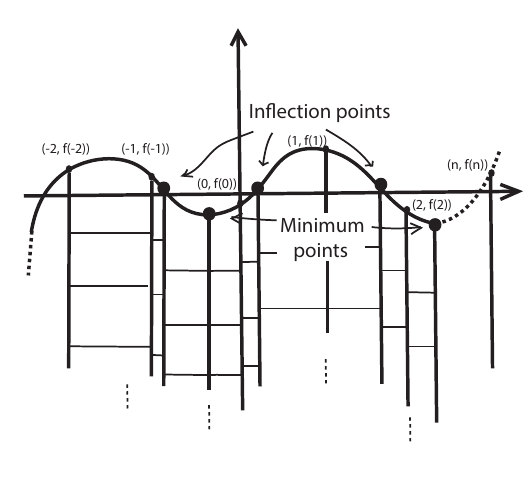} 
\caption{A cell decomposition of $X$ when $deg(f) > 1$.}
\label{fig:16}
\end{figure}

First, we look at the 0-cells. Let $e_0$ be a 0-cell. If $\gamma$ passes through $e_0$ and comes back to it later, we can replace $\gamma$ with a strictly shorter piecewise $C^2$-curve by letting it stay at $e_0$. Thus $\gamma$ interacts with each 0-cell at most once. 

Second, we look at the 2-cells. Let $e_2$ be a 2-cell. If the closure $\overline{e}_2$ of $e_2$ does not lie on the graph, $\overline{e}_2$ is a rectangle. By the convexity of $e_2$, $\gamma$ interacts with $e_2$ at most once. Suppose $\overline{e}_2$ has one side lying on the graph, it suffices to show that $\gamma$ interacts with this side at most finitely many times, then we can conclude that $\gamma$ interacts with $e_2$ at most finitely many times. We will prove that this is true as follows. 
 
Let $e_1$ be a 1-cell. Suppose $e_1$ is on the graph, there are three cases in general: 
\begin{enumerate}
\item $f'' < 0$ on $e_1$, so the cell is strictly convex downward; 
\item $f'' > 0$, $f' > 0$ on $e_1$, so the cell is strictly convex upward and also strictly increasing; 
\item $f'' > 0$, $f' < 0$ on $e_1$, so the cell is strictly convex upward and strictly decreasing. 
\end{enumerate}
We only need to treat the first and second cases separately, because for the third, we can apply a reflectional symmetry to get to the second case.

{\it Case One: $f'' < 0$ on $e_1$}. Consider the closure $\overline{e}_1$ of $e_1$. If $C$ is the leftmost point in the intersection of $\overline{e}_1$ and $\gamma$, and $D$ is the rightmost point, the straight line segment from $C$ to $D$ is contained in $X$, because $\overline{e}_1$ is convex downward. Therefore, the part of $\gamma$ from $C$ to $D$ is a straight line segment. If $C, D$ are both inside $e_1$, $\gamma$ interacts $e_1$ twice; if at least one of $C, D$ is a boundary point of $e_1$, it interacts $e_1$ at most once. Thus $\gamma$ interacts at most twice with $e_1$.

{\it Case Two: $f'' > 0, f' > 0$ on $e_1$}. Consider the closure $\overline{e}_1$ of $e_1$ again, and let $C, D$ be the same two points as before. We ask the questions: what is the part of $\gamma$ from $C$ to $D$? Is it the part of the graph from $C$ to $D$? Is it unique? The answers are yes. Let's start proving them.
First we need to study two preliminary things: zigzag tangent curves and approximation of shortest-length curves by zigzag tangent curves. 

\subsection{Zigzag tangent curves}
\label{subsec:2.1}
A polygonal curve is a piecewise linear curve. Given two points $A$, $B$ in $\overline{e}_1$, assume $\gamma$ is any arbitrary polygonal curve between the two points, which lie in $X$, then we can demonstrate one optimization of $\gamma$ so that the new polygonal curve is shorter, more regular, and has no more vertices than $\gamma$. There are two steps shown as follows:

{\it Step one:} We replace $\gamma$ with a shorter polygonal curve that is tangent to $f(x)$ at $A$ and $B$. Let the tangent line to $f(x)$ at $A$ be called $l$. Note $e_1$ is entirely above $l$, except at the point $A$. If the slope of $\gamma$ at $A$ is less than the slope of $l$, $\gamma$ ducks under $l$ at the beginning. Because $B$ is above $l$, $\gamma$ has to cross $l$ later in order to end up at $B$. Let's call $C$ an intersection point of $\gamma$ and $l$. Then we can replace the part of $\gamma$ from $A$ to $C$ with the line segment between them, thus yielding a shorter polygonal curve. On the other hand, if the slope of $\gamma$ is greater than that of $l$, $\gamma$ is above the graph of $f$ for points near $A$. Thus $\gamma$ goes outside of the region $X$, this is a contradiction to our hypothesis. Similarly, we can replace $\gamma$ with a shorter polygonal curve that is tangent to $f(x)$ at $B$. 
 
{\it Step two:} Let's first look at a motivation for this step.  Let $C$ be the intersection point of the two tangent lines to $f(x)$ at $A$ and $B$ (see Figure~\ref{fig:2}). Then $A - C - B$ is the shortest among all possible polygonal curves under the graph of $f(x)$ in the form of $A - C' - B$, where $C'$ is any arbitrary point.  This is proved in the following lemma. 
 
\begin{figure}
\includegraphics[width=5cm]
{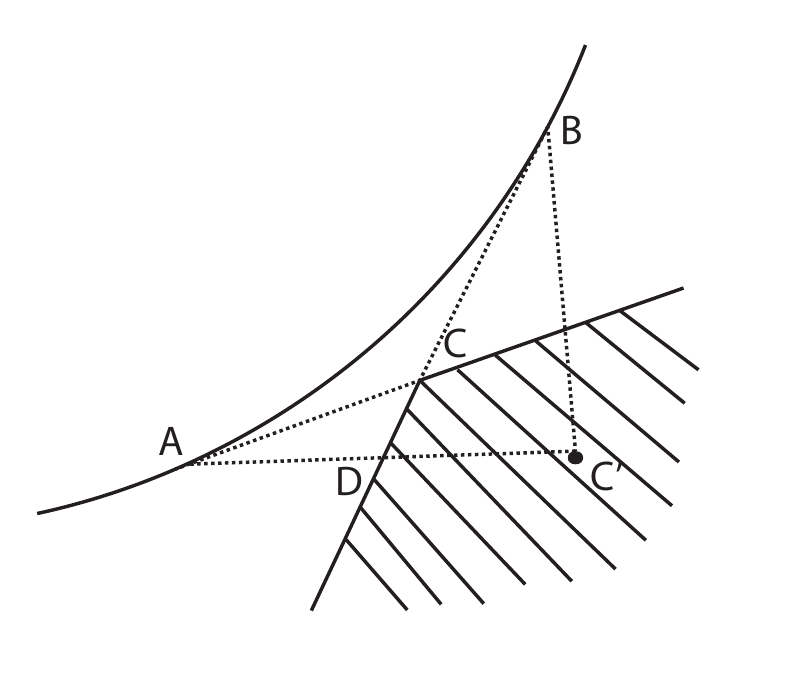} 
\caption{$|AC| + |BC| \leq |AC'| + |BC'|$.}
\label{fig:2}
\end{figure}

\begin{lemma}
\label{lem:1}
If $A-C-B$ is a polygonal curve such that $AC$ and $BC$ are tangent to the graph of $f(x)$ at $A$ and $B$, respectively, and $C'$ is any other arbitrary point such that the polygonal curve $A-C'-B$ is under the graph of $f(x)$, then $|AC| + |CB| \leq |AC'| + |C'B|$. 
\end{lemma}

\begin{proof}
Since $A-C'-B$ is under the graph of $f(x)$, $AC'$ needs to stay below the line passing through $AC$, i.e. the slope of $AC'$ is smaller than that of $AC$. Similarly, $BC'$ needs to stay below the line passing through $BC$. Therefore we find out the region where $C'$ could be. It is shown as the shaded region in Figure~\ref{fig:2}. 

Next let's compare the length of $A - C - B$ with that of $A - C' - B$. Suppose $D$ is the intersection point of  $AC'$ with the line passing through $BC$, then we can apply the triangle inequality to get the following inequalities:
\begin{equation*}
|AC'| + |C'B| = |AD| + (|DC'| + |C'B|) \geq |AD| + |BD| 
 = (|AD| + |DC|) + |CB| \geq |AC| + |BC|,
\end{equation*}
as desired. 
\end{proof}

If we replace our original $\gamma$ with $A - C- B$, the length of $A-C-B$ is not necessarily less than or equal to that of $\gamma$. In order to fix this problem, we want to replace $\gamma$ with a new polygonal curve that looks like a repetition of $A-C-B$. That is to say, the new curve touches the graph of $f(x)$ at every other vertex. Moreover, at each of those vertices, the new curve is tangent to the graph of $f$ as shown in Figure~\ref{fig:35}. We call such a curve {\bf a zigzag tangent curve}. 

\begin{figure}[ht]
\includegraphics[width=5cm]
{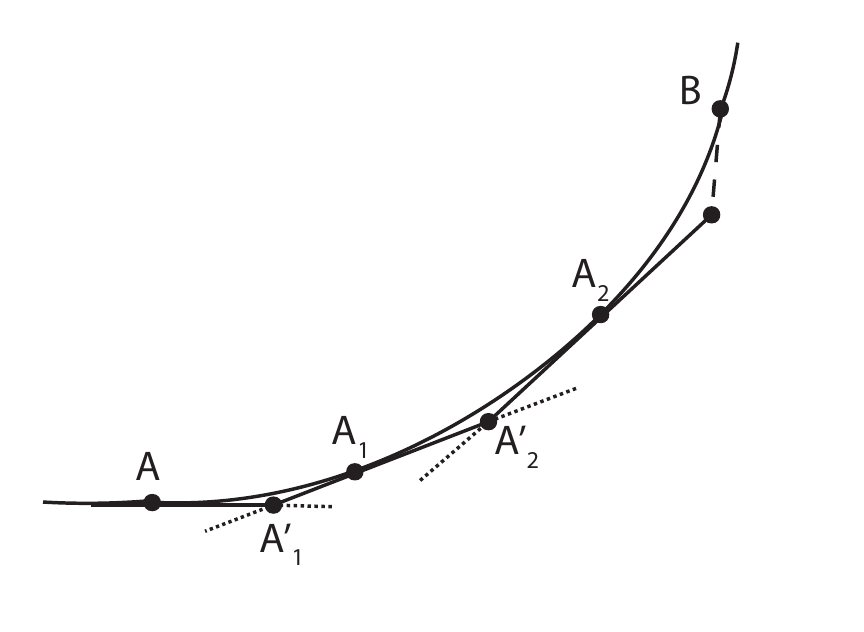} 
\caption{A zigzag tangent curve from $A$ to $B$ that is tangent to $f(x)$ at $A_1, A_2, \ldots, B$.}
\label{fig:35}
\end{figure}

The motivation is as follows. First, we start off from $A$. Next, at vertex $A_1'$, we adjust the direction. Hoping to reach $B$ as shortest as possible under the restriction of staying below the graph, it seems that the best direction to go next is in the direction of a line that is tangent to the graph of $f$ and towards $B$. So $A_1$ touches the graph. Then, after passing through the tangent point $A_1$, we adjust our direction again at $A_2'$. Keep repeating this procedure finitely many steps until we reach $B$. So a zigzag tangent curve seems to be the ``best" candidate to replace any polygonal curve from $A$ to $B$. 
 
\begin{lemma}
\label{lem:2}
Suppose $\gamma$ is any arbitrary polygonal curve from $A$ to $B$, where $A, B$ are points in the closure of some strictly convex upward and strictly increasing 1-cell $e_1$. Moreover, assume $\gamma$ is tangent to the graph of $f(x)$ at $A$ and $B$. Then $\gamma$ can be replaced by a zigzag tangent curve such that its length is no bigger than that of $\gamma$ and it also has no more vertices than $\gamma$. 
\end{lemma}

\begin{proof}
We induct on the number $n$ of intermediate vertices between $A$ and $B$ on $\gamma$. If $n=1$, $\gamma$ is already the zigzag tangent curve $A - C - B$ shown in Figure~\ref{fig:2}. If $n \geq 2$, let $A_1, \ldots, A_n$ be the intermediate vertices between $A$ and $B$. If we look at Figure~\ref{fig:3}, we see that as $D$ slides from $A$ to $B$ on the graph, its tangent line to $f(x)$ intersects with $AC$ in a point that varies continuously from $A$ to $C$. So by the intermediate value theorem, if $A_1$ is on $AC$, there exists a point $D$ on the graph between $A$ and $B$ such that its tangent line to $f(x)$ intersects with $AC$ at $A_1$. On the other hand, if $D$ is above $B$, its tangent line does not intersect with $AC$ at all. Here we need to consider two cases.

\begin{figure}
\includegraphics[width=5cm]
{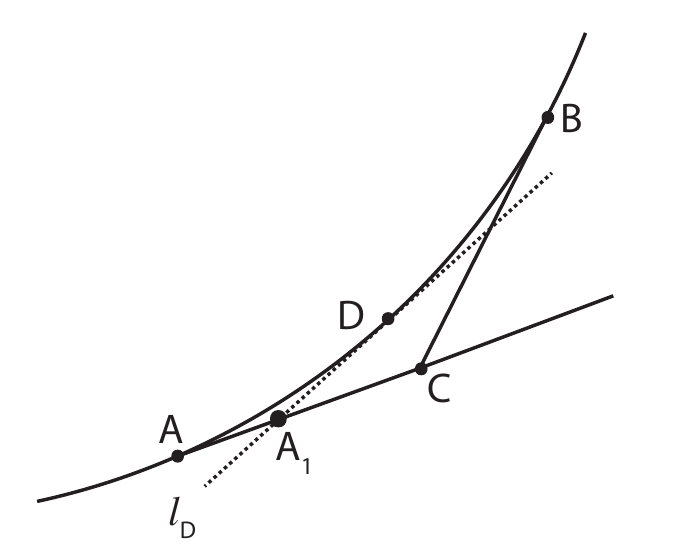} 
\caption{The existence of a tangent line to the graph pass through $A_1$.}
\label{fig:3}
\end{figure}

{\it Case 1}: If $A_1$ is not on $AC$, we simple use $A - C - B$ as our new polygonal curve.  Since $C$ and $B$ are both on $\gamma$ and the shortest path between them is $BC$, the length of $A - C - B$ is shorter than that of $\gamma$ (see Figure~\ref{fig:4}). Similarly, we do the same thing if $A_1 = C$, in which $A - C - B$ is no longer than $\gamma$. In both cases, $A - C - B$ has no more vertices than $\gamma$.  

\begin{figure}[ht]
\includegraphics[width=5cm]
{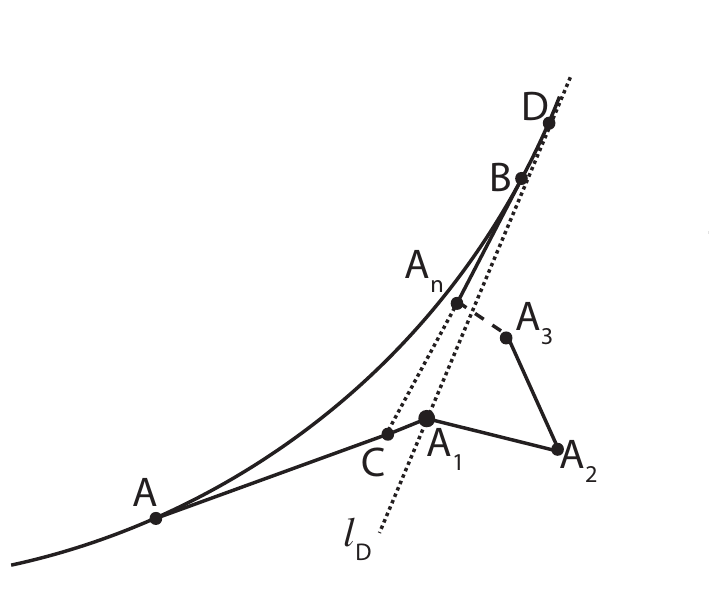} 
\caption{$A_1$ is not on $AC$.}
\label{fig:4}
\end{figure}

{\it Case 2}: Assume that $A_1$ is on $AC$, but not $C$. Let the tangent line to $f(x)$ at $D$ be $l_D$. Then $B$ lies above $l_D$. There are two subcases in this case.
 
{\it Subcase 1}: If $A_2$ is below $l_D$, $\gamma$ must cross $l_D$ at a point, say $A_1'$, in order to reach $B$. One example is shown in Figure~\ref{fig:5}. Suppose $A_1'$ lies on the line segment between $A_{k-1}$ and $A_k$, then $k \geq 3$. Replace $\gamma$ with the new polygonal curve $A - A_1 - A_1' - A_k - A_{k+1} - \ldots - A_n -B$. Call the new curve $\gamma'$. For example in Figure~\ref{fig:5}, $\gamma' = A - A_1 - A_1' - A_4 - A_5 - \ldots - A_n -B$, which has one less vertex than $\gamma$ and is strictly shorter than $\gamma$. In general, the number of vertices of $\gamma'$ is less than or equal to that of $\gamma$, and the length of $\gamma'$ is less than that of $\gamma$. 
\begin{figure}
\includegraphics[width=5cm]
{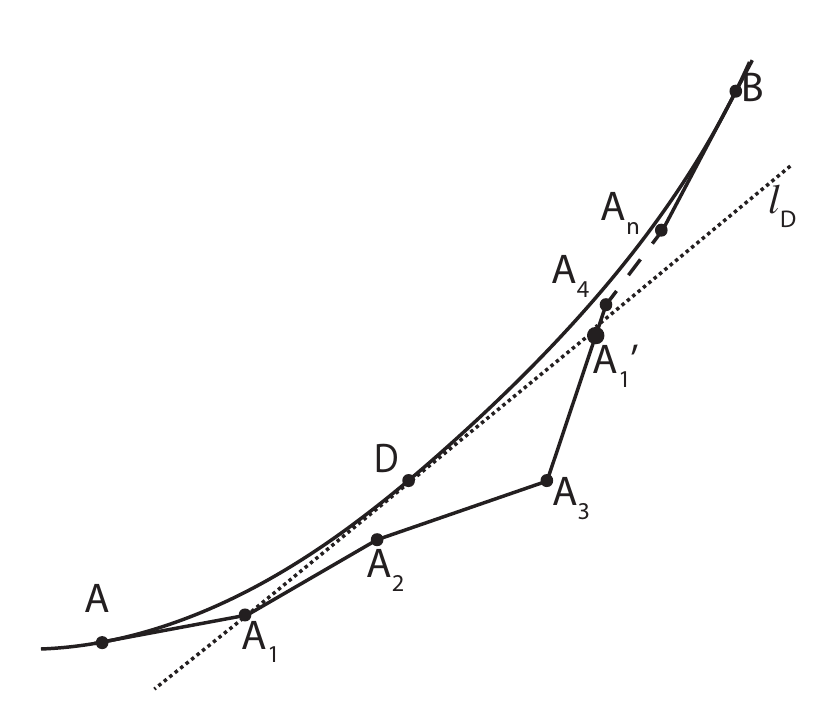} 
\caption{$A_2$ lies below the tangent line passing through $D$.}
\label{fig:5}
\end{figure}
We can divide $\gamma'$ into two parts: one part is $A - A_1 - D$ and the other part is $ D - A_1' - A_k - A_{k+1} - \ldots - A_n - B$. Since $D - A_1' - A_k - A_{k+1} - \ldots - A_n - B$ is a polygonal curve that is tangent to the graph at the endpoints and has less than $n$ intermediate vertices, we can apply the inductive hypothesis to replace it with a zigzag tangent curve with no bigger length and no more vertices. Combining it with $A - A_1 - D$, we get a desired zigzag tangent curve for $\gamma$.

{\it Subcase 2}:
If $A_2$ lies on $l_D$, we need to consider the following three separate subsubcases. 

{\it Subsubcase 1}: When $A_2$ is to the right of $D$ as shown in Figure~\ref{fig:6}, there are $n-1$ vertices between $D$ and $B$, so we can apply the inductive hypothesis to the polygonal curve $D - A_2 - A_3 - \ldots - B$. 

\begin{figure}[ht]
\includegraphics[width=5cm]
{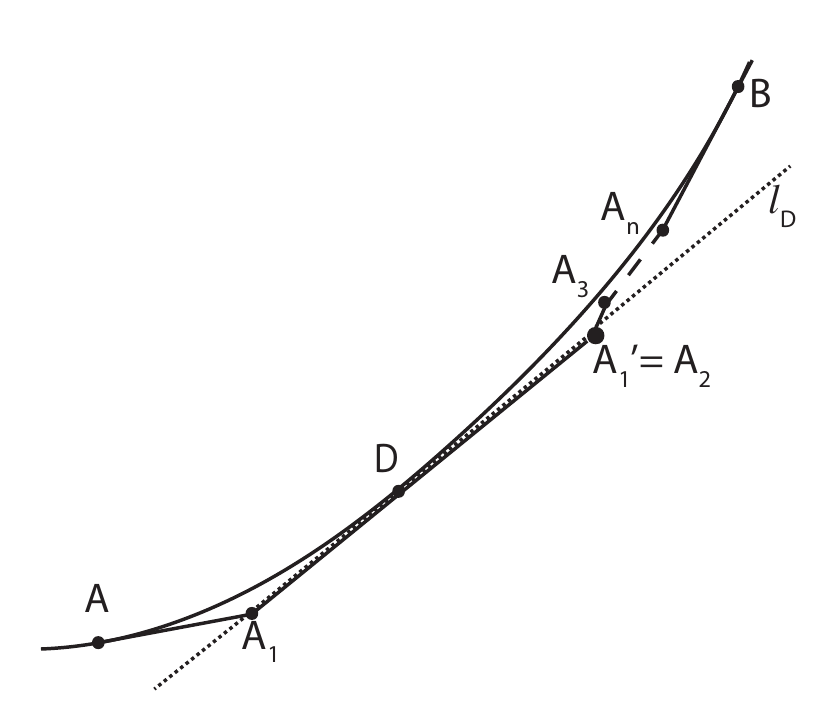} 
\caption{$A_2$ lies to the right of $D$ on the tangent line passing through $D$.}
\label{fig:6}
\end{figure}

{\it Subsubcase 2}: When $A_2$ is exactly $D$ as shown in Figure~\ref{fig:7}, $l_D$ intersects with $\gamma$ at $A_2$ and at least one other point. This is because $A_3$ is either below $l_D$ or on it, in either case $l_D$ intersects $\gamma$ at another point besides $A_2$. Let's call this point $A_1'$, and $A_1' = A_3$ when $A_3$ is on $l_D$. Let the new curve be $A - A_1- D - A_1' - A_k - A_{k+1} -\ldots - A_n - B$, where $k \geq 4$. There are at most $n-2$ vertices from $D$ to $B$, thus the inductive hypothesis applies. 

\begin{figure}[ht]
\includegraphics[width=5cm]
{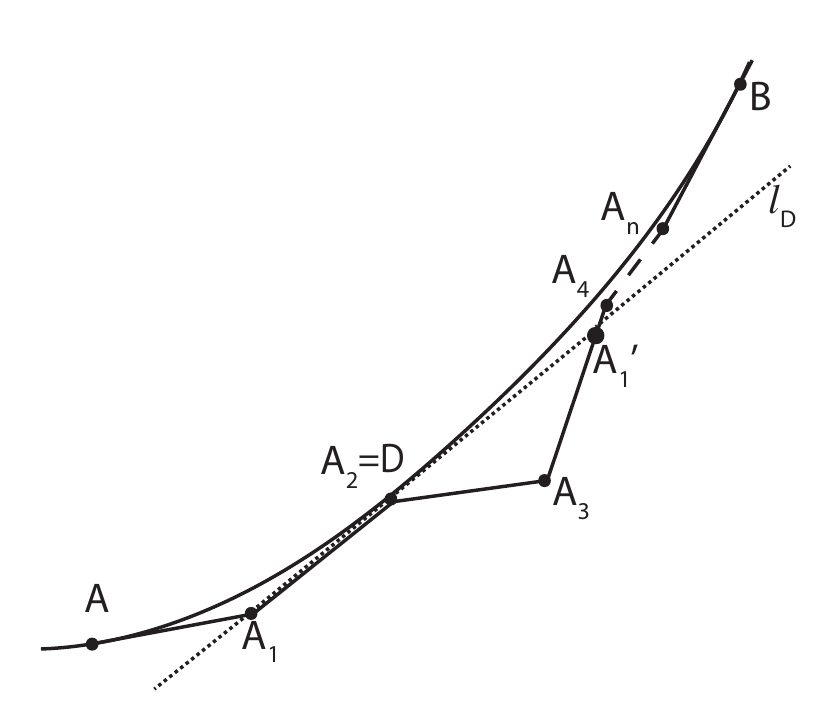} 
\caption{$A_2$ lies exactly on $D$ on the tangent line passing through $D$.}
\label{fig:7}
\end{figure}

{\it Subsubcase 3}: When $A_2$ is to the left of $D$ as shown in Figure~\ref{fig:8}, similar to the previous argument, $l_D$ crosses $\gamma$ at a point $A_1'$ other than $A_2$, and there are at most $n - 2$ vertices between $D$ and $B$. 

\begin{figure}[ht]
\includegraphics[width=5cm]
{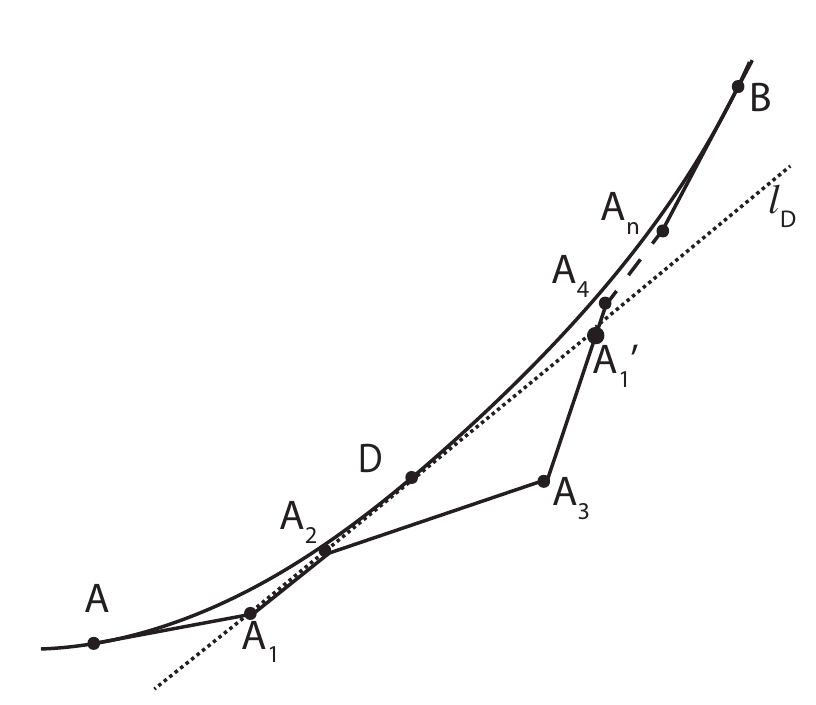} 
\caption{$A_2$ lies to the left of $D$ on the tangent line passing through $D$.}
\label{fig:8}
\end{figure}

\end{proof}

As a summary, we obtain the following result.
\begin{theorem}
\label{thm:1}
If $\gamma$ is a polygonal curve from $A$ to $B$, where $A, B$ lie in the closure of a 1-cell $e_1$. Assume on $e_1$, $f'' > 0$ and $f' > 0$, then $\gamma$ can be replaced by a zigzag tangent curve from $A$ to $B$ which is no longer than $\gamma$ and has no more vertices than $\gamma$. 
\end{theorem}

\subsection{Approximation of shortest-length curves by polygonal curves}
\label{subsec:2.2}
Assume $\gamma$ is a piecewise $C^2$-curve with endpoints in $\overline{e}_1$ and $f'' > 0, f' > 0$ on $e_1$ as before. Our goal is to approximate $\gamma$ by zigzag tangent curves below the graph whose lengths can be made as close to that of $\gamma$ as possible. 

Let's parametrize $\overline{e}_1$ by a smooth function $\alpha$
\begin{equation*}
\alpha: [x_1, x_2] \rightarrow \mathbb{R}^2, \alpha(x) = (x, f(x)), 
\end{equation*}
where $x_1, x_2$ are suitably chosen so that $f''(x) > 0, f'(x) > 0$ for all $x \in (x_1, x_2)$. Given $\epsilon > 0$, there exists $\eta > 0$ such that if $Q = \{x_1=s_0 < s_1 < \ldots < s_m = x_2 \}$ is a partition with $\max_{1 \leq j \leq m} |s_j - s_{j-1}| < \eta$, then 
\begin{equation*}
\big| L_{x_1}^{x_2}(\alpha, Q) - \int_{x_1}^{x_2} |\alpha'(s)| ds \big| < \epsilon,
\end{equation*}
where $L_{x_1}^{x_2}(\alpha, Q) = \sum_{1}^{m} |\alpha(s_i) - \alpha(s_{i-1})|$ is the length of the polygonal curve inscribed in $\overline{e}_1$ with respect to partition $Q$ (see [1]). Similarly for $\gamma: [a,b] \rightarrow \mathbb{R}^2$, if $P = \{a = t_0 < t_1 < \ldots < t_n = b\}$ is a partition of the interval $[a, b]$, then we can find $\delta > 0$ such that $\max_{1 \leq i \leq n} |t_i - t_{i-1}| < \delta$ implies
\begin{equation*}
\big|L_a^b(\gamma, P) - \int_a^b |\gamma'(t)| dt\big| < \epsilon.
\end{equation*}
Furthermore, we can choose $\delta$ small enough so that $\max_{1 \leq i \leq n} |\gamma(t_i) - \gamma(t_{i-1})| < \eta$ due to the continuity of $\gamma$.

Now let's focus on the polygonal curve inscribed in $\gamma$ with partition $P$. For convenience, we also call it $P$. We don't know whether each line segment in $P$ lies in $X$ or not, but we can fix this by modifying the part of $P$ that is above the graph through the following two steps. 

{\it Step one:} We replace each part of $P$ that is above the graph by the corresponding part of the graph below it, and show that the difference in length between the old and new curves is less than $\epsilon$.  This is better understood by first looking at the Figure~\ref{fig:9}.
\begin{figure}[ht]
\includegraphics[width=5cm]
{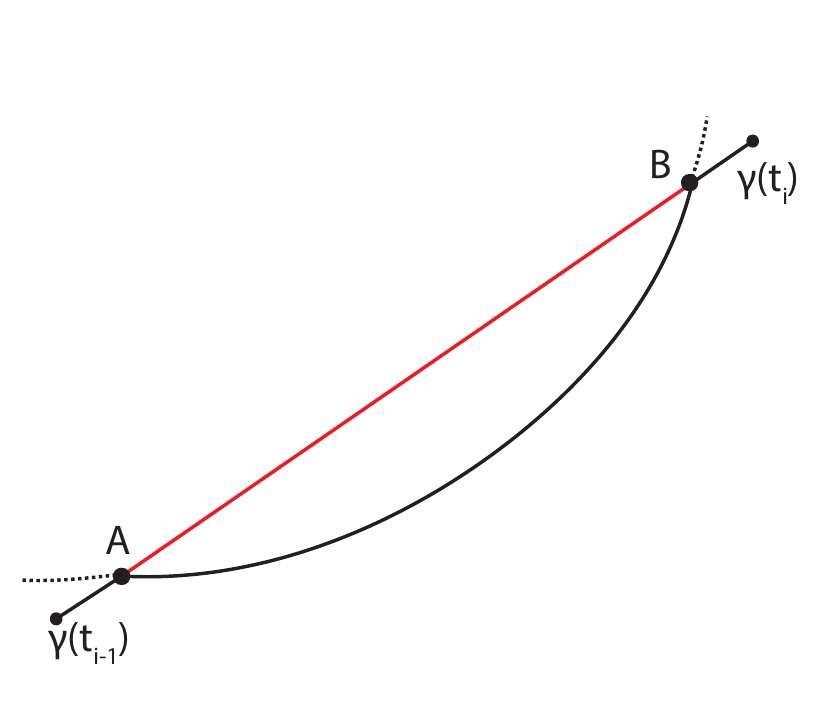} 
\caption{Replace the part of the polygonal line that is above the graph by the graph below it.}
\label{fig:9}
\end{figure}
Suppose part of the line segment between $\gamma(t_{i-1})$ and $\gamma(t_i)$ is above the graph, which is shown in red in the figure. We label the endpoints of this line segment as $A$ and $B$. After replacing the line segment with the graph of $f(x)$ from $A$ to $B$ below it, we obtain a new curve between $\gamma(t_{i-1})$ and $\gamma(t_i)$, which is shown in solid black in the figure. We do this for every $i$ from 1 to $n$ and call the resulting new curve $\tilde{\gamma}$. We want to show that $\tilde{\gamma}$ has a length that is within $\epsilon$ range of $L_{a}^{b}(\gamma, P)$.  

According to the way $\delta$ was chosen, if part of the line segment from $\gamma(t_{i-1})$ and $\gamma(t_i)$ is above the graph as shown in Figure~\ref{fig:9}, then $|AB| \leq |\gamma(t_{i-1}) - \gamma(t_i)| < \eta$. We can create a partition $Q$ by connecting those line segments which are above the graph together as shown in Figure~\ref{fig:10}. 
\begin{figure}[ht]
\includegraphics[width=5cm]
{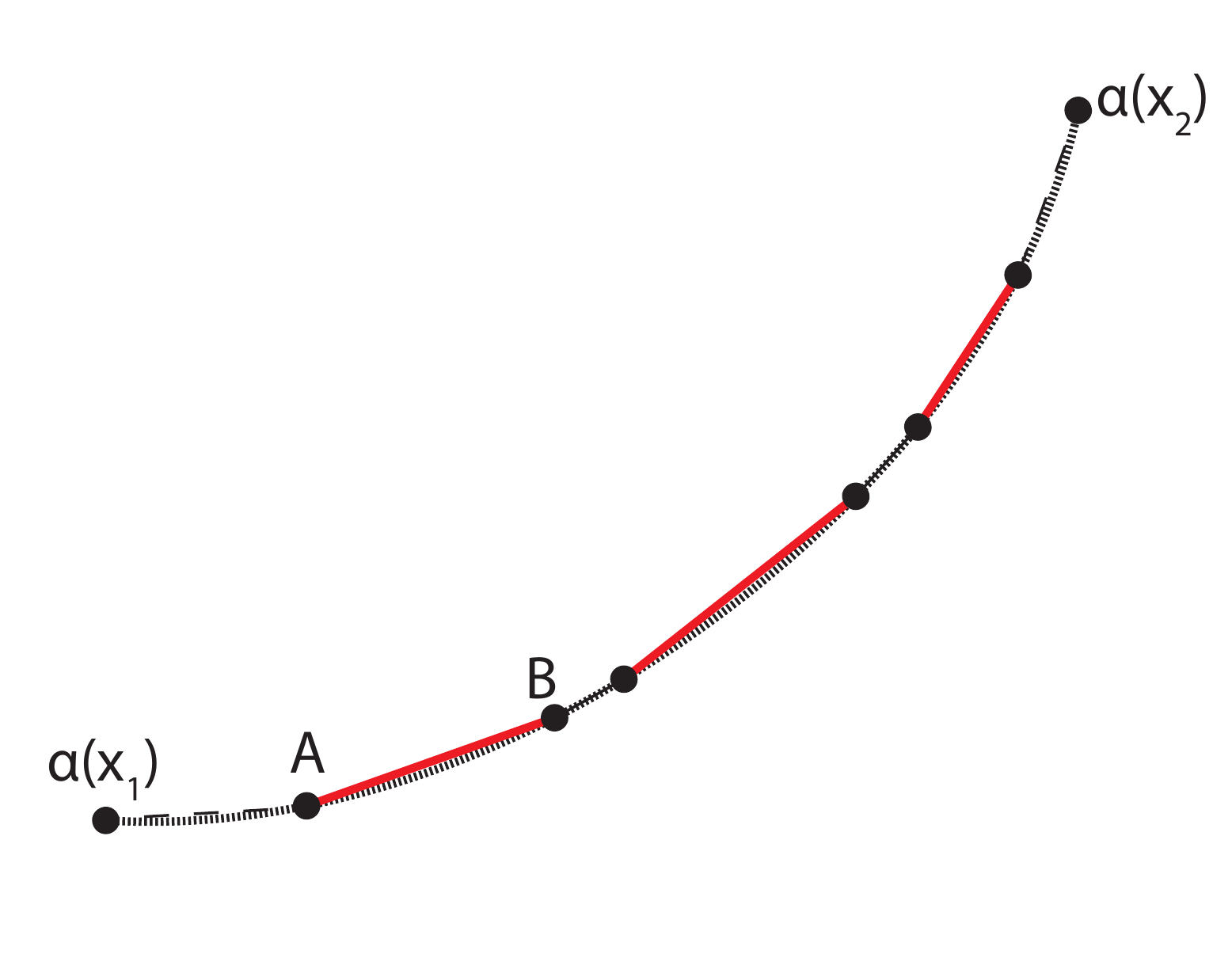} 
\caption{Create a polygonal curve that includes all line segments (shown in red) above the graph between the two endpoints of $\overline{e}_1$.}
\label{fig:10}
\end{figure}
It follows that 
\begin{equation}
\label{eqn:7}
0 \leq  \int_{a}^{b} |\tilde{\gamma}'(t)| dt - L_{a}^{b}(\gamma, P) \leq \int_{x_1}^{x_2} |\alpha'(s)| ds - L_{x_1}^{x_2}(\alpha, Q)  <\epsilon,
\end{equation}
where $\int_{a}^{b} |\tilde{\gamma}'(t)| dt - L_{a}^{b}(\gamma, P)$ is the sum of all differences, such as the difference between the line segment from $A$ to $B$ and the graph from $A$ to $B$ in Figure~\ref{fig:9}. Furthermore, we add the following inequality:
\begin{equation*}
-\epsilon < L_{a}^{b}(\gamma, P) -  \int_{a}^{b} |\gamma'(t)| dt < 0
\end{equation*}
to (2.1) and obtain that the length of $\tilde{\gamma}$ is within $\epsilon$ range of that of $\gamma$ as follows: 
\begin{equation}
\label{eqn:1}
-\epsilon < \int_a^b |\tilde{\gamma}'(t)| dt - \int_a^b |\gamma'(t)| dt < \epsilon.
\end{equation}

{\it Remark}: If you were careful enough, you might have spotted an error in our argument above. That is we were assuming that the line segments above the graph are not crossing each other, except possibly at the endpoints. If there were crossings, then we weren't able to connect the line segments above the graph in order to construct $Q$. So we need the following lemma to make sure that this situation does not happen. 

\begin{lemma}
\label{lem:3}
It $\gamma$ does not move back and forth in horizontal or vertical direction, then for any polygonal curve $P$ inscribed in $\gamma$, the line segments above the graph do not overlap, except possibly at the endpoints. 
\end{lemma}

\begin{proof}
Without loss of generality we assume that the initial point of $\gamma$ is to the left of the final point of $\gamma$. Since $\gamma$ does not move back and forth in the horizontal direction, the $x$-coordinate function of $\gamma$ is increasing. The $y$-coordinate function of $\gamma$ is either increasing or decreasing, depending upon whether the initial point is above or below the final point.  The graph of $\gamma$ looks like one of the two as shown in Figure~\ref{fig:11}. Notice that there can be at most countable many of these vertical line segments in $\gamma$. 

\begin{figure}[ht]
\includegraphics[width=5cm]
{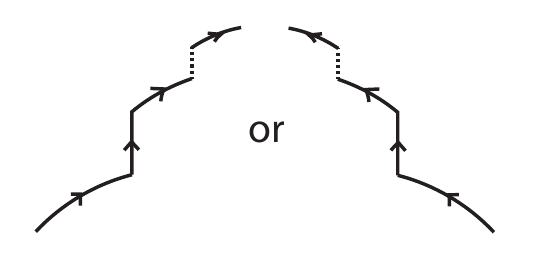}
\caption{Assume $\gamma$ moves from left to right, $\gamma$ either goes up or goes down in the vertical direction.}
\label{fig:11}
\end{figure}

Without loss of generality, we assume $\gamma$ is going up as it moves from left to right. Suppose $P$ is an arbitrary polygonal curve inscribed in $\gamma$. Let $n + 1$ be the number of vertices in $P$ and we apply the mathematical induction on $n$. Name the vertices of $P$ as $A_0, A_1, \ldots, A_n$. When $n=1$, $P$ is a line segment, so the lemma is trivially true. When $n \geq 2$, by inductive hypothesis, there is no crossing up to the vertex $A_{n-1}$. $A_n$ is either vertically above, or to the right of $A_{n-1}$. Since the line segment between $A_{i-1}$ and $A_i$, for $1 \leq i \leq n-1$, is either to the left or vertically below $A_{n-1}$, the line segment from $A_{n-1}$ to $A_n$ does not cross any of them, except possibly at the endpoint $A_{n-1}$.  
\end{proof}

We can make Lemma 2.4 as part of our assumption for $\gamma$, because we are concerned with shortest-length curves and shortest-length curves do not move back and forth in horizontal or vertical direction. (If a piecewise $C^2$-curve does move back and forth, we can always replace it with another one with a strictly shorter length.) 

{\it Step two:} We replace each part of $\tilde{\gamma}$ which is on the graph by a zigzag tangent curve below it. Call the new curve $\tilde{\tilde{\gamma}}$, then $\tilde{\tilde{\gamma}}$ is a polygonal curve below the graph. We want to verify that the difference in length between $\tilde{\gamma}$ and $\tilde{\tilde{\gamma}}$ is less than $\epsilon$. Therefore after combining with (2.2), the difference in length between $\gamma$ and $\tilde{\tilde{\gamma}}$ is less than $2\epsilon$. 

Since there are only finitely many parts of $\tilde{\gamma}$ on the graph, we may assume without loss of generality that $\tilde{\gamma}$ has only one such part, namely the part from $A$ to $B$ as shown in Figure~\ref{fig:9}. 
We may assume that the part of the graph from $A$ to $B$ is $\overline{e}_1$, then it is parametrized by $\alpha$
\begin{equation*}
\alpha: [x_1, x_2] \rightarrow \mathbb{R}^2, \alpha (x) = (x, f(x)),
\end{equation*}
where $\alpha(x_1) = A$ and $\alpha(x_2) = B$. If $Q = \{x_1=s_0 < s_1 < \ldots < s_m = x_2 \}$ is a partition of $[x_0, x_1]$, we construct a zigzag tangent curve associated to those points. Recall that a zigzag tangent curve looks like a repetition of $A - C - B$ in Figure~\ref{fig:3}. To obtain such a curve, we first draw a tangent line at each point $\alpha(s_j)$ for $0 \leq j \leq m$, then every pair of consecutive tangent lines intersect at a point. Connecting those points together, we get a zigzag tangent curve (see Figure~\ref{fig:12}). 
\begin{figure}[ht]
\includegraphics[width=5cm]
{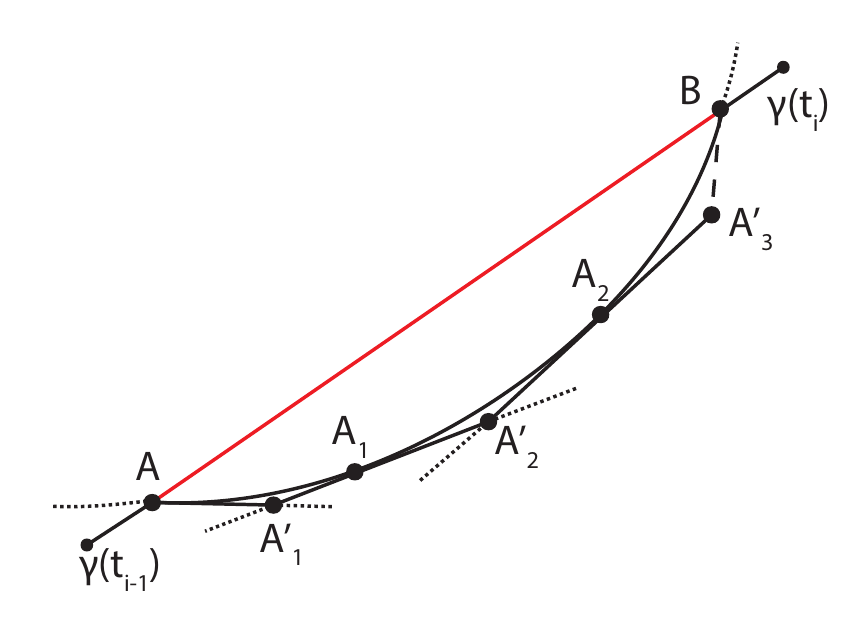} 
\caption{Approximate the graph from $A$ to $B$ by a zigzag tangent curve.}
\label{fig:12}
\end{figure}
 Let's denote $\alpha(s_j)$ by $A_j$ and the intersection point between $A_{j-1}$ and $A_j$ by $A_j'$, and thus the zigzag tangent curve is the polygonal line $A - A_1' - A_1 - A_2' - A_2 - \ldots - B$. We denote the length of the zigzag tangent curve as $L_{x_1}^{x_2}(\alpha, ZQ)$. 

If we could also make the zigzag tangent curve to have a length close to that of the polygonal curve $A - A_1 - A_2 - \ldots - B$, then the length of the zigzag tangent curve is close to that of $\alpha$, because we can make the polygonal curve to have a length close to that of $\alpha$. This is what we need. 

\begin{lemma}
\label{lem:4}
Assume $\alpha: [x_1, x_2] \rightarrow \mathbb{R}^2, \alpha (x) = (x, f(x))$, satisfies $\alpha(x_1) = A$ and $\alpha(x_2) = B$, and $Q = \{x_1=s_0 < s_1 < \ldots < s_m = x_2 \}$ is a partition of $[x_1, x_2]$. Then for any $\epsilon > 0$, there exists $\eta > 0$ such that if $\max_{1 \leq j \leq m} |s_j - s_{j-1}| < \eta$, then
\begin{equation*}
\big| \int_{x_1}^{x_2} |\alpha'(s)| ds - L_{x_1}^{x_2} (\alpha, Q) \big| < \epsilon,  \text{ and }  \big| L_{x_1}^{x_2}(\alpha, ZQ) - L_{x_1}^{x_2}(\alpha, Q)\big| < \epsilon. 
\end{equation*}
Consequently, 
\begin{equation*}
\big| \int_{x_1}^{x_2} |\alpha'(s)| ds - L_{x_1}^{x_2} (\alpha, ZQ) \big| < 2\epsilon.
\end{equation*}
\end{lemma}

\begin{proof}
The first inequality is clear, as we've already seen it few times before. Now let's prove the second inequality. The key idea here is to complete each zigzag into a triangle as shown in Figure~\ref{fig:13}. 
\begin{figure}[ht]
\includegraphics[width=4cm]
{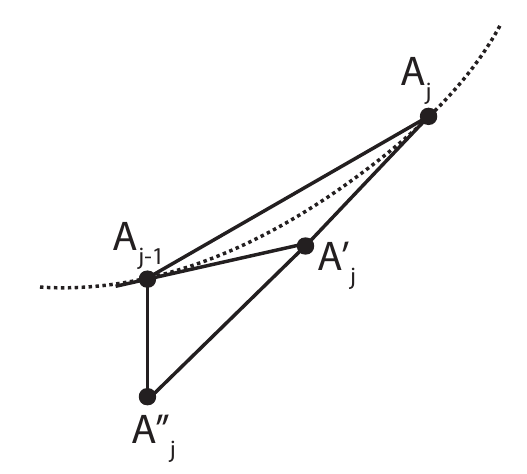} 
\caption{Complete each zigzag into a triangle.}
\label{fig:13}
\end{figure}

Assume that for the zigzag between $A_{j-1}$ and $A_j$, we extend $A_j A_{j}'$ to intersect with the vertical line passing through $A_{j-1}$ at a point $A_{j}''$. It follows that
\begin{equation*}
|A_{j-1}A_j'| + |A_j'A_j| - |A_{j-1}A_j| < |A_{j-1}A_j''| + |A_j''A_j| - |A_{j-1}A_j|.
\end{equation*} 
Let the length of the new polygonal curve $A - A_1'' - A_1 - \ldots - A_{j-1} - A_j'' - A_j - \ldots - B$ be called $L_{x_1}^{x_2} (\alpha, \triangle Q)$, so we get 
\begin{equation*}
0 < L_{x_1}^{x_2}(\alpha, ZQ) - L_{x_1}^{x_2}(\alpha, Q) < L_{x_1}^{x_2}(\alpha, \triangle Q) - L_{x_1}^{x_2}(\alpha, Q). 
\end{equation*}
Therefore for the second inequality in the lemma, it suffices to show
\begin{equation*}
L_{x_1}^{x_2}(\alpha, \triangle Q) - L_{x_1}^{x_2}(\alpha, Q) < \epsilon,
\end{equation*}
if $\max_{1 \leq j \leq m} |s_j - s_{j-1}| < \eta$.

For each $j$, the tangent line at $A_j$ has slope $f'(s_j)$, so the length of $A_j'' A_j$ is 
\begin{equation*}
|A_j'' A_j|=(s_j - s_{j-1})\sqrt{1 + f'(s_j)^2}. 
\end{equation*}
Since the equation of the tangent line is $y = f(s_j) + f'(s_j)(x - s_j)$, the length of $A_{j-1} A_j ''$ is
\begin{equation*}
f(s_{j-1}) - f(s_j) - f'(s_j)(s_{j-1} - s_j).
\end{equation*}
It implies that 
\begin{eqnarray*}
\lefteqn{L_{x_1}^{x_2}(\alpha, \triangle Q) - L_{x_1}^{x_2}(\alpha, Q)}\\
& = &  \sum_{j=1}^m f(s_{j-1}) - f(s_j) - f'(s_j)(s_{j-1} - s_j) \\
 & & \mbox{}  + (s_j - s_{j-1})\sqrt{1 + f'(s_j)^2} - \sqrt{(s_j - s_{j-1})^2 + (f(s_j) - f(s_{j-1}))^2}  \\
& =  & \sum_{j=1}^m (s_j - s_{j-1})\big\{ f'(s_j) - \frac{f(s_j - s_{j-1})}{s_j - s_{j-1}} + \sqrt{1 + f'(s_j)^2}\\
&   & \mbox {} - \sqrt{1+ \big[\frac{f(s_j - s_{j-1})}{s_j - s_{j-1}}\big]^2}\big\}.
\end{eqnarray*}
By the mean value theorem, there exists $s_{j-1} < \hat{s}_j < s_j$ such that $f(s_j - s_{j-1}) = f'(\hat{s}_j)( s_j - s_{j-1})$. Plugging into the above equation, we obtain
\begin{eqnarray*}
\lefteqn{L_{x_1}^{x_2}(\alpha, \triangle Q) - L_{x_1}^{x_2}(\alpha, Q)} \\
& = & \sum_{j=1}^m (s_j - s_{j-1})\big\{ f'(s_j) - f'(\hat{s}_j) + \sqrt{1 + f'(s_j)^2} - \sqrt{1+ (f'(\hat{s}_j)^2}\big\}.
\end{eqnarray*}
Now let's define the function $ g(x) = f'(x) + \sqrt{1 + f'(x)^2} $, so we can rewrite the above equality as 
\begin{eqnarray}
\label{eqn:2}
L_{x_1}^{x_2}(\alpha, \triangle Q) - L_{x_1}^{x_2}(\alpha, Q) = \sum_{j=1}^m (s_j - s_{j-1})\big\{ g(s_j) - g(\hat{s}_j)\big\}.
\end{eqnarray}
Since $g$ is uniformly continuous on $[x_1, x_2]$, given $\epsilon > 0$, we can choose $\eta$ small enough such that if $\max_{1 \leq j \leq m} |s_j - s_{j-1}| < \eta$, 
\begin{equation*}
g(s_j) - g(\hat{s}_j) < \epsilon / (x_2 - x_1)
\end{equation*}
for each $1 \leq j \leq m$. Therefore (2.3) implies that 
\begin{equation*}
0 < L_{x_1}^{x_2}(\alpha, \triangle Q) - L_{x_1}^{x_2}(\alpha, Q) < \epsilon,
\end{equation*}
as desired.
\end{proof}

Putting steps one and two together, the following proposition is a formal statement of approximating shortest-length curves by polygonal curves.
\begin{proposition}
\label{prop:1}
Suppose $f(x)$ is a polynomial function of degree $\geq 2$, and $X$ is the closed region below the graph of $f(x)$. Let a cell decomposition of $X$ be given as shown in Figure~\ref{fig:16}, and let $e_1$ be a 1-cell on the graph of $f(x)$, on which $f'' > 0$ and $f' >0$. Assume $\gamma$ is a piecewise $C^2$-curve in $X$ whose initial and final points lie inside the closure of $e_1$, and $\gamma$ does not move back and forth in horizontal or vertical direction, then for any $\epsilon > 0$, there exists a polygonal curve $\tilde{\tilde{\gamma}}$ in $X$ such that the difference in length between $\gamma$ and $\tilde{\tilde{\gamma}}$ is within $\epsilon$.  
\end{proposition}

More generally, we can show that this is also true for any $\gamma$.
\begin{proposition}
\label{prop:2}
Let $f(x)$ be a polynomial function of degree $\geq 2$, and let $X$ be the closed region below the graph of $f(x)$. Assume $\gamma: [a, b] \rightarrow \mathbb{R}^2$ is a piecewise $C^2$-curve between two points in $X$ and $\gamma$ does not move back and forth in horizontal or vertical direction, then given any $\epsilon > 0$, there exists a polygonal curve $\tilde{\tilde{\gamma}}$ in $X$ such that the difference in length between $\gamma$ and $\tilde{\tilde{\gamma}}$ is within $\epsilon$.  
\end{proposition}

\begin{proof}
Since the graph above $\gamma$ might be convex upward or downward, we want to first divide $\gamma$ by drawing vertical lines through the 0-cells on the graph, namely $(n, f(n))$, $n \in \mathbb{Z}$, and the strict inflection points and the local minimum points. As a result, each part of $\gamma$ is below one of the three types of 1-cells on the graph. Recall the three types are where: $f'' < 0$, $f'' > 0$ and $f' > 0$, $f'' > 0$ and $f' < 0$. 

Given $\epsilon > 0$, there exists $\delta > 0$ so that for any partition $P = \{a = t_0 < t_1 < \ldots < t_n = b\}$ of $[a, b]$, if $\max_{1 \leq i \leq n} |t_i - t_{i-1}| < \delta$, then
\begin{equation*}
\big|L_a^b(\gamma, P) - \int_a^b |\gamma'(t)| dt\big| < \epsilon.
\end{equation*}
We may add more points to $P$ so that each line segment in the polygonal curve is also below one of the three types of 1-cells. 

For each $1 \leq i \leq n$, if the line segment between $\gamma(t_{i-1})$ and $\gamma(t_i)$ is below a 1-cell where $f'' < 0$, then it is contained in $X$; if it is below a 1-cell where $f'' > 0, f' > 0$, we can approximate it with another polygonal curve below the graph according to the previous proposition; if it is below a 1-cell where $f'' > 0, f' < 0$, the argument is similar to when $f'' > 0, f' > 0$. Therefore we obtain a polygonal curve $\tilde{\tilde{\gamma}}$ that is below the graph whose length is within $\epsilon$ range of that of $\gamma$.
\end{proof}

This proposition immediately gives us the following corollary. 
\begin{corollary}
\label{cor:1}
Given any two points $A$ and $B$ in the closed region $X$ below the graph of a polynomial function $f(x)$,  if there exists a shortest-length curve $\gamma$ from $A$ to $B$ in $X$, then 
\begin{eqnarray}
\label{eqn:3}
\inf \{ \int |\zeta'|: \mbox{$\zeta$ is a piecewise $C^2$-curve from $A$ to $B$ in $X$}\} = \nonumber \\
\inf \{ \int |\zeta'| : \mbox{$\zeta$ is a polygonal curve from $A$ to $B$ in $X$}\}.
\end{eqnarray}
\end{corollary}

\begin{proof}
%$\inf \{ \int |\zeta'|: \mbox{$\zeta$ is a piecewise $C^2$-curve from $A$ to $B$ in $X$}\}$ is realized by $\gamma$, and it implies that
First, $\gamma$ does not move back and forth in horizontal or vertical direction, otherwise we can replace it with another curve of strictly less length. Second, from Proposition 2.7, we can approximate $\gamma$ by a sequence of polygonal curves in $X$ whose lengths decrease to that of $\gamma$. So the equality holds. 
\end{proof}

Now we are ready to formally answer the questions that were asked earlier: If $C$, $D$ are two points in the closure of a 1-cell $e_1$, where $f''>0$ and $f'>0$, and $\gamma$ is a shortest curve from $C$ to $D$ under the graph, what does $\gamma$ look like? Is it unique?
\begin{proposition}
\label{prop:3}
Suppose $f(x)$ is a polynomial function of deg $\geq 2$, and $X$ is the closed region below the graph of $f(x)$. Let a cell decomposition of $X$ be given as shown in Figure 3, and let $e_1$ be a 1-cell on the graph of $f(x)$, where $f'' > 0$ and $f' > 0$. Assume that $\gamma$ is a shortest-length curve between two points $C, D$ in the closure of $e_1$, then $\gamma$ lies entirely on the graph of $f(x)$ from $C$ to $D$. 
\end{proposition}

\begin{proof}
({\it Existence}): By Corollary 2.8, it suffices to study polygonal curves in $X$ from $C$ to $D$. Let $\zeta$ be such a curve. Then Theorem 2.3 implies that $\zeta$ can be replaced by a zigzag tangent curve $\zeta_1$ from $C$ to $D$ which is no longer than $\zeta$. Suppose $\zeta_1$ is given by
\begin{equation*}
C - A_1' - A_1 - \ldots - A_{j-1} - A_j' - A_j - \dots - A_n' - A_n - D,
\end{equation*}
where $\zeta_1$ is tangent to $f(x)$ at $A_j$ for $1 \leq j \leq n$. We can add another tangent point $B_j$ between $A_{j-1}$ and $A_j$, then the tangent line at $B_j$ intersects $A_{j-1}A_j'$ and $A_j'A_j$ at the new points $B_j'$ and $B_j''$, respectively. See Figure~\ref{fig:36}. 
\begin{figure}[ht]
\includegraphics[width=5cm]
{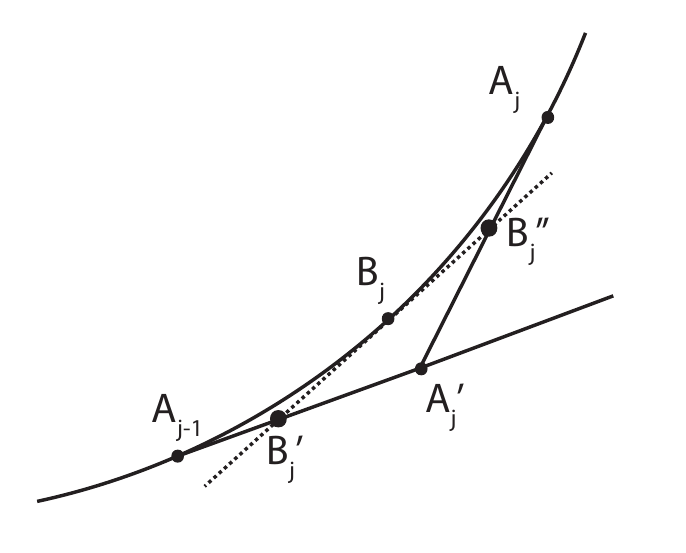} 
\caption{Decrease the length of a zigzag tangent curve by adding more tangent points.}
\label{fig:36}
\end{figure}
Call the new curve $\zeta_2$. By the triangle inequality, $ |B_j'B_j''| < |B_j'A_j'| + |A_j'B_j''|$,
so $\zeta_2$  has a length smaller than that of $\zeta_1$. Continue this process and we get a sequence $\{\zeta_k\}$ of zigzag tangent curves between $C$ and $D$ whose lengths are strictly decreasing. We want to show that the limit is the length of the graph from $C$ to $D$. 

Given $\epsilon > 0$, since $\gamma$ is a shortest-length curve in $X$ from $C$ to $D$, Proposition 2.6 implies that there exists a polygonal curve, say $\zeta$, such that 
\begin{equation*}
\int |\zeta'| - \epsilon < \int |\gamma'| \leq \int |\zeta'|,
\end{equation*}
which implies that 
\begin{equation}
\label{eqn:8}
\int |\zeta'| - \epsilon < \int |\gamma'| \leq \int |\zeta_k'|,
\end{equation}
for any $k \geq 1$.

By Lemma 2.5, the difference in length between a zigzag tangent curve and the graph can be made arbitrarily small if the tangent points on the zigzag tangent curve are close enough to each other. That is to say, when $k$ is sufficiently large, 
\begin{equation*}
\big| \text{length of the graph from $C$ to $D$} - \int |\zeta_k'| \big| < \epsilon,
\end{equation*}
equivalently it can be rewritten as follows:
\begin{equation}
\label{eqn:9}
- \int |\zeta_k'| - \epsilon < - (\text{length of the graph from $C$ to $D$}) < - \int |\zeta_k'| + \epsilon.
\end{equation}
Since $\int |\zeta_k'| < \int |\zeta'|$, (2.6) gives 
\begin{equation}
\label{eqn:10}
- \int |\zeta'| - \epsilon < - (\text{length of the graph from $C$ to $D$}) < - \int |\zeta_k'| + \epsilon,
\end{equation}
Adding (2.5) and (2.7) yields
\begin{equation*}
-2\epsilon < \int |\gamma'| - \text{length of the graph from $C$ to $D$} < \epsilon.
\end{equation*}
This is true for any arbitrary $\epsilon$. Therefore the length of the graph from $C$ to $D$ is equal to that of $\gamma$, and so the graph from $C$ to $D$ is also a shortest curve in $X$ from $C$ to $D$. Moreover, the sequence $\{\zeta_k\}$ satisfies that their lengths decrease to the length of the graph from $C$ to $D$. 

({\it Uniqueness}): Suppose not, then there is a point $I$ on $\gamma$ that is not on the graph. Then $I$ has an open neighborhood in which the line segment between any two points in the neighborhood is below the graph. It implies that near $I$, $\gamma$ must be linear, otherwise we can pick two points where $\gamma$ is not linear in between and get a shorter curve by replacing the part with a line segment. We can extend the linear curve near $I$ on both ends until each end hits the graph of $f(x)$. This must be true because the initial and final points of $\gamma$ are on the graph. In the end, we obtain a line segment below the graph with two endpoints on the graph, which is a contradiction because the graph is convex upward. 
\end{proof}

{\it Remark}: The proof for uniqueness gives another way of proving the proposition, which is much shorter. However, it was discovered much later. We will implement this idea and the zigzag tangent curve in the future for higher-dimensional cases. Continue our line of thinking, we get the following statement.
\begin{theorem}
\label{thm:2}
Suppose $f(x)$ is a polynomial function, $X$ is the closed region below the graph of $f(x)$, and $A, B$ are two arbitrary points inside $X$. Without loss of generality, we assume that $A$ is to the left of $B$. Assume $\gamma: [a, b] \rightarrow \mathbb{R}^2$ is a piecewise $C^2$-curve from $A$ to $B$ that is a shortest-length curve from $A$ to $B$ in $X$. Then there exists a cell decomposition of $X$ such that $\gamma$ interacts with each cell at most finitely many times. Moreover, it interacts with each 0- or 1-cell at most twice. 
\end{theorem}

{\it Remark}: There is possibly a triangulation theorem using our construction of a cell decomposition.

\section{More general regions in the plane}
\label{sec:3}
\subsection{Type I regions}
\label{subsec:3.1}
We think of a simply-connected region whose {\bf sides} consist of finitely many graphs of polynomial functions. Three examples are shown in Figure~\ref{fig:14}. Let's still denote such a region by $X$. Previously $X$ has only one side being the graph of a polynomial function. Here, $X$ could have more than one side, $X$ could be bounded or unbounded, and the boundary of $X$ could be disconnected. We call such an $X$ {\bf a region of type I} and $X$ has vertices and sides as shown in Figure~\ref{fig:14}. We ask the same question as before: does there exist a cell decomposition of $X$ such that any shortest-length curve between two points in $X$ interacts with each cell at most finitely many times? The answer is yes! 

\begin{figure}[ht]
\includegraphics[width=9cm]
{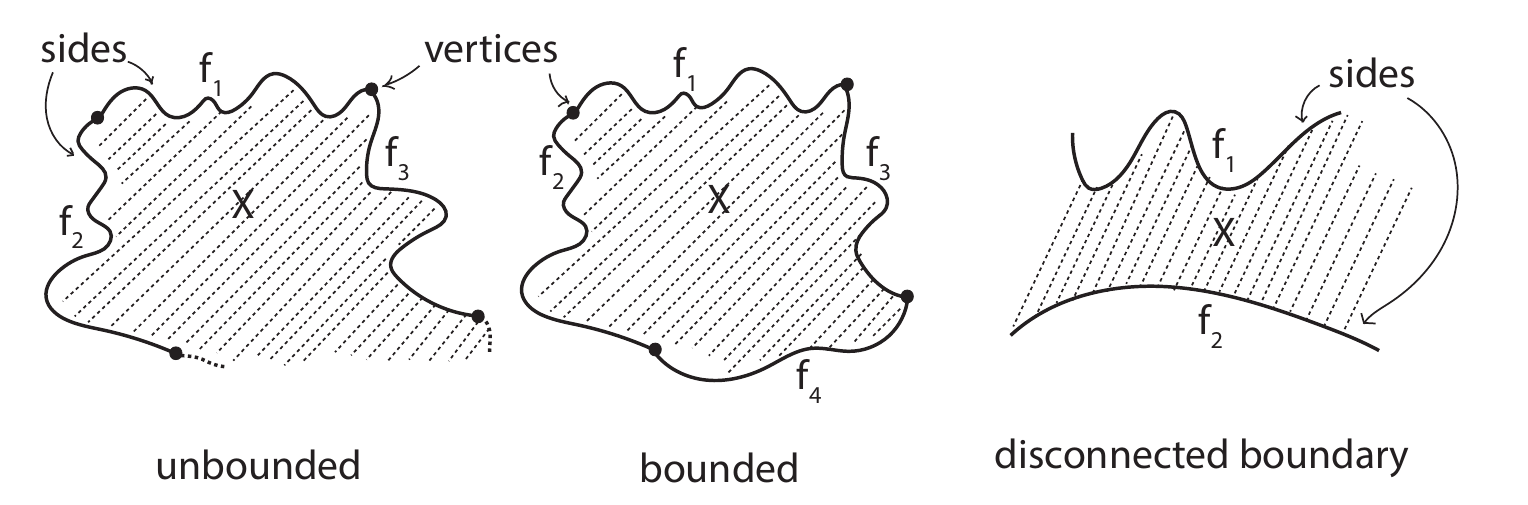} 
\caption{Three examples of regions of type I.}
\label{fig:14}
\end{figure}

\begin{proposition}
\label{prop:4}
Assume $X$ is a region of type I, then $X$ has a cell decomposition such that any shortest-length curve between two points in $X$ interacts with each cell at most finitely many times. 
\end{proposition}

\begin{proof}
First of all, let's construct a cell decomposition of $X$. Since each side of $X$ is the graph of a polynomial function and there are finitely many of them, we name the polynomials $f_1(x), f_2(x), \ldots, f_k(x)$. If $k = 1$, we return to the familiar situation in section~\ref{sec:2}. For our next purpose, we call such a region {\bf a polynomial half plane}. The name comes from the observation that such a region looks very much like a half plane, except that the boundary is a polynomial curve instead of a straight line. Furthermore, in the convention, if we rotate or translate a half plane, we still call it a half plane, so we do the same thing for a polynomial half plane. 

Next, if $k \geq 2$, the intuition for constructing a cell decomposition for $X$ is as follows. Let $X_1, \ldots, X_k$ be the corresponding polynomial half planes for $f_1(x), \ldots, f_k(x)$, respectively. For each $X_i$, we use the same decomposition as shown in Figure~\ref{fig:16}. It is possible that we need to add more 0-cells to the graph of $f_i$ in this decomposition, such as the vertices of $X$. After dividing each $X_i$ into cells, we overlay all of them to obtain a cell decomposition for $X$.

\begin{lemma}
\label{lem:5}
The construction above yields a cell decomposition for $X$. In addition, the cell decomposition looks like a web of meshes; each mesh is enclosed by finitely many edges; and each edge is either a line segment or a section of one side of $X$ . 
\end{lemma}
{\it Remark}: We call each mesh in the lemma {\bf a generalized polygon} or {\bf a mesh}, so a generalized polygon is a bounded region of type I as shown in Figure~\ref{fig:14}. 

\begin{proof}
We prove by induction on $k$. When $k=1$, each generalized polygon has four edges with at least three edges being linear; moreover, if it has a nonlinear edge, it must be a section of the graph of $f_1(x)$. When $k \geq 2$, let $X' = X_1 \cap X_2 \cap \ldots \cap X_{k-1}$. By inductive hypothesis, overlaying the cell decompositions of $X_1, X_2, \ldots, X_{k-1}$ gives  a cell decomposition of $X'$ satisfying the property that the closure of every 2-cell is a generalized polygon with each edge being either linear or a section of the graph of $f_1(x), f_2(x), \ldots$, or $f_{k-1}(x)$. Let $e_2$ be such a 2-cell in the cell decomposition. For $X_k$, it also has a grid-like cell decomposition satisfying the property that each edge is either linear or a section of the graph of $f_k(x)$. 

Consider the web in $X$ that is obtained from overlaying the two cell decompositions of $X'$ and $X_k$. We want to show that it divides $X$ into generalized polygons, thus is a cell decomposition of $X$. We need to consider two cases.

{\it Case 1}: If $\overline{e}_2$ is completely inside $X_k$, since $\overline{e}_2$ is compact and each mesh in $X_k$ has a minimum area, $\overline{e}_2$ is covered by only finitely many generalized polygons in $X_k$.

{\it Claim}: $\overline{e}_2$ is divided into finitely many generalized polygons after overlaying the grids of $X'$ and $X_k$.

{\it Proof of Claim}: First, finitely many vertices of those generalized polygons covering $\overline{e}_2$ are contained in $\overline{e}_2$. 

Second, given one edge of $\overline{e}_2$, every edge of those meshes covering $\overline{e}_2$ either overlaps with it partially or intersects it at finitely many points. This is shown as follows. Let $\omega_1$ be an edge of $\overline{e}_2$, and let $\omega_2$ be an edge of any of the meshes that covers $\overline{e}_2$. When both $\omega_1$ and $\omega_2$ are linear, they either overlap partially or intersect at most at one point. When one of them is nonlinear, say $\omega_2$, then $\omega_2$ is on the graph of $f_i(x)$ for some $1 \leq i \leq k-1$. In the coordinate frame of $f_i(x)$, we can write the equation of the graph on which $\omega_1$ lies as follows: 
\begin{equation*}
g(\cos(\theta) x + \sin(\theta) y) = - \sin(\theta) x + \cos(\theta) y,
\end{equation*}
where $g(x)$ is either a linear function or $f_k(x)$, and $\theta$ is the angle between the coordinate frames of $f_i(x)$ and $g(x)$. Then for the intersection points, we solve the following equation: 
\begin{equation*}
g(\cos(\theta) x + \sin(\theta) f_i(x)) = - \sin(\theta) x + \cos(\theta) f_i(x)
\end{equation*}
This is a polynomial function in $x$. If we get $0 = 0$, $\omega_1$ and $\omega_2$ overlap partially, otherwise we get finitely many zeros. In total, there are finitely many intersection points of edges in those meshes and edges in $\overline{e}_2$. 

Third, let $D$ be the set of vertices of those meshes in $X_k$ covering $\overline{e}_2$ and the intersection points of edges in those meshes and edges in $\overline{e}_2$. From the previous two steps, $D$ is a finite set. Furthermore, we observe that for each pair of points in this collection, they are connected by at most one curve which is either from an edge of $\overline{e}_2$ or from one belonging to one of those meshes in $X_k$. Since the generalized polygons covering $\overline{e}_2$ do not overlap, except possibly on the edges, $\overline{e}_2$ is divided into finitely many meshes, whose edges are on either $\overline{e}_2$ or one of the generalized polygons covering $\overline{e}_2$, and so are either linear or on the graphs of $f_1(x), f_2(x), \ldots, f_k(x)$. This finishes the proof for the claim, thus completing the case 1.

{\it Case 2}: If $\overline{e}_2$ is partially contained in $X_k$, then $\overline{e}_2$ is first cut off by the graph of $f_k(x)$. We claim that the graph of $f_k$ divides $e_2$ into finitely many generalized polygons. The proof of this claim will be given in Theorem 3.6. For $X$, we select only these generalized polygons that are contained in $X_k$. It follows that each of these generalized polygons is further divided up by the meshes of $X_k$ into finitely many generalized polygons using the same argument as the previous case. 
\end{proof}

Let's continue proving the proposition. Given the cell decomposition as in Lemma 3.2, suppose $\gamma$ is a shortest-length curve between two points in $X$, $\gamma$ interacts with each 0-cell at most once. For the 1-cells, let $e_1$ be one of these. By Lemma 3.2, $e_1$ is either linear or a section of the graph of $f_1(x), f_2(x), \ldots$, or $f_k(x)$. If $e_1$ is linear, $\gamma$ interacts with it at most once. If $e_1$ is a section of the graph of $f_1(x), f_2(x), \ldots$, or $f_k(x)$, there are three cases as follows. 

{\it Case 1}: $e_1$ is linear, then $\gamma$ interacts with $e_1$ at most once.

{\it Case 2}: $e_1$ is convex upward, then $\gamma$ also interacts with $e_1$ at most once according to Proposition 2.9. 

{\it Case 3}: $e_1$ is convex downward, then $\gamma$ interacts with $e_1$ at most finitely many times. Suppose not, take one point from each interaction, then we have an infinite set $E$ of these points. By the compactness of $\overline{e}_1$, $E$ has an accumulation point $x$ in $\overline{e}_1$, where $x$ is either a boundary point of $e_1$ or not. Since $\gamma$ is a closed curve, $x$ is also in $\gamma$. Without loss of generality, we assume $\gamma$ interacts with $e_1$ infinitely many times before reaching $x$, otherwise we can reverse the direction of $\gamma$. Here we need to separate into three subcases as follows.

{\it Subcase 1}: $x$ is not a boundary point of $e_1$. Since $x$ is at a positive distance away from every other side of $X$, there exists $r > 0$ such that $X$ contains every point below $e_1$ which is at a distance less than $r$ from $x$ (see Figure~\ref{fig:21}). Since $x$ is an accumulation point of $E$, we can pick three distinct points $x_1, x_2, x_3$ in $E$ that are at distances less than $r$ from $x$.  Thus the line segment between every pair of them is contained in $X$. Without loss of generality, assume $\gamma$ goes through the three points in order of $x_1, x_2, x_3$. Then the shortest path from $x_1$ to $x_3$ is the straight line segment, which does not intersect $e_1$ in between. This is a contradiction to the fact that $\gamma$ has another interaction with $e_1$ at $x_2$ between $x_1$ and $x_3$.

\begin{figure}[ht]
\includegraphics[width=3cm]
{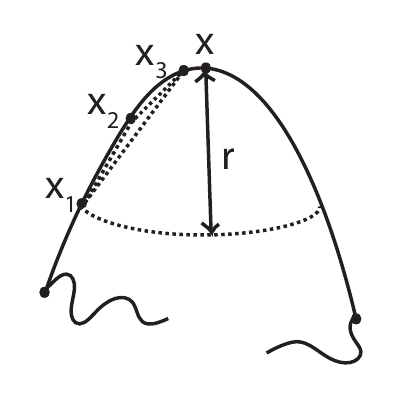} 
\caption{When $x$ is on $e_1$, we can pick three distinct points that are in $B(x,r)$.}
\label{fig:21}
\end{figure}

{\it Subcase 2}: $x$ is a boundary point of $e_1$. If $x$ is not a boundary point of the side containing $e_1$, we can use the same argument as above. Now suppose $x$ is a boundary point of the side containing $e_1$, equivalently, $x$ is a vertex point of $X$. Let the first 1-cell on the next side be $\tilde{e}_1$. Then $\tilde{e}_1$ also has $x$ as a boundary point. Assume $e_1$ and $\tilde{e}_1$ are oriented in the clockwise direction. Denote the tangent lines of $e_1, \tilde{e}_1$ at $x$ as $l_1, l_2$, respectively. Then there are two subsubcases as below.

{\it Subsubcase 1}: When $l_1$ and $l_2$ do not coincide, we show the following lemma is true. 
\begin{lemma}
\label{lem:6}
Assume the tangent lines $l_1$ and $l_2$ do not coincide. Then when a point on $e_1$ is close enough to $x$, the line segment between them is contained in $X$. 
\end{lemma}

\begin{proof}
First, suppose $\tilde{e}_1$ is convex upward in its own coordinate frame, there are two possible configurations as shown in Figure~\ref{fig:22}. 
\begin{figure}[ht]
\includegraphics[width=9cm]
{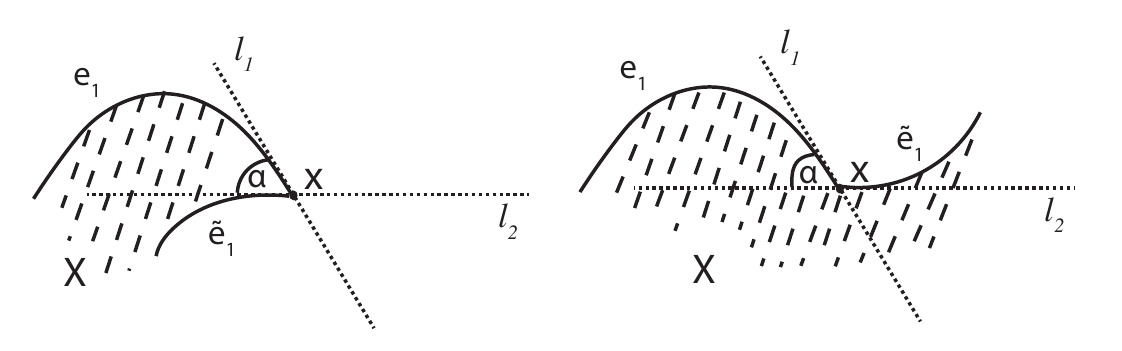} 
\caption{$l_1 \neq l_2$ and $\tilde{e}_1$ is convex upward in its own coordinate frame.}
\label{fig:22}
\end{figure}
Let $\alpha$ be the angle from $l_1$ to $l_2$ in the counterclockwise direction, then $0 < \alpha < \pi$. Given a point $y$ on $e_1$, if $y$ is close enough to $x$, the angle between the line segment $[x, y]$ and $l_1$ is less than $\alpha$, thus $[x, y]$ is above $l_2$. Since $[x, y]$ is below $e_1$, $[x, y]$ is inside $X$. 

Second, suppose $\tilde{e}_1$ is convex downward in its own coordinate frame, there are also two possible configurations as shown in Figure~\ref{fig:23}. If $\tilde{e}_1$ is below $l_2$, we can use the same argument as before. On the other hand, if $\tilde{e}_1$ is above $l_2$, $\tilde{e}_1$ crosses $e_1$ at another point, say $a$, besides $x$. Then for every point $y$ on $e_1$ that is between $a$ and $x$, the line segment $[x, y]$ is above $\tilde{e}_1$ and below $e_1$, and thus is in $X$. 

\begin{figure}[ht]
\includegraphics[width=9cm]
{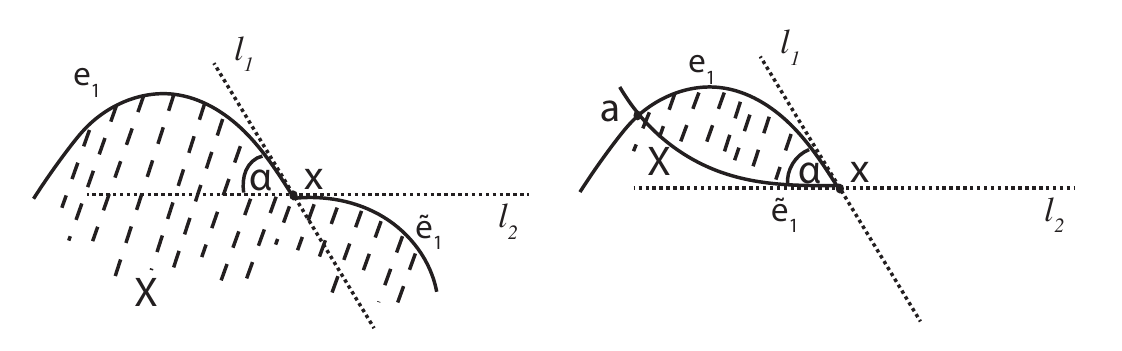} 
\caption{$l_1 \neq l_2$ and $\tilde{e}_1$ is convex downward in its own coordinate frame.}
\label{fig:23}
\end{figure}
\end{proof}

According to Lemma 3.3, when $l_1$ and $l_2$ do not coincide with each other, there exists a point $y$ such that $\gamma$ passes $y$ before $x$ and the line segment $[x, y]$ lies in $X$. Therefore $\gamma$ is a straight line before reaching $x$, which is a contradiction to our assumption that before arriving at $x$, $\gamma$ interacts with $e_1$ at infinitely many points converging to $x$. 

{\it Subsubcase 2}: When $l_1$ and $l_2$ do coincide, there are two possibilities.

First, when $\tilde{e}_1$ is convex downward in its own coordinate frame, there is only one possible configuration as shown in Figure~\ref{fig:20}. Draw the perpendicular line $l'$ to $l_1$ at $x$. Using a similar argument as in Lemma 3.3, when a point on $e_1$ is close enough to $x$, the line segment between them is above $l'$ and below $e_1$, thus inside $X$. Thus it follows that $\gamma$ is again a straight line before reaching $x$. So we get a contradiction. 

\begin{figure}[ht]
\includegraphics[width=4cm]
{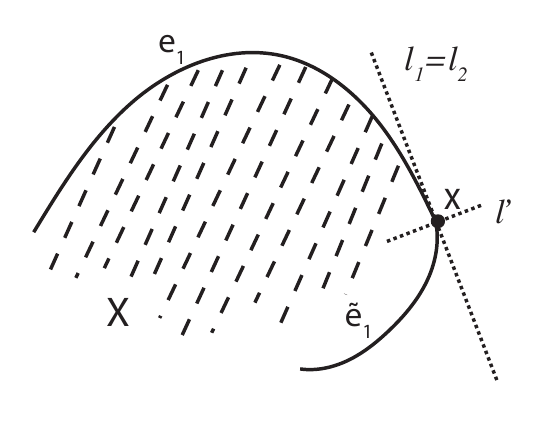} 
\caption{$l_1 = l_2$ and $\tilde{e}_1$ is convex downward in its own coordinate frame..}
\label{fig:20}
\end{figure} 

Second, when $\tilde{e}_1$ is convex upward in its own coordinate frame, there are two possible configurations as shown in Figure~\ref{fig:24}.
\begin{figure}
\includegraphics[width=8cm]
{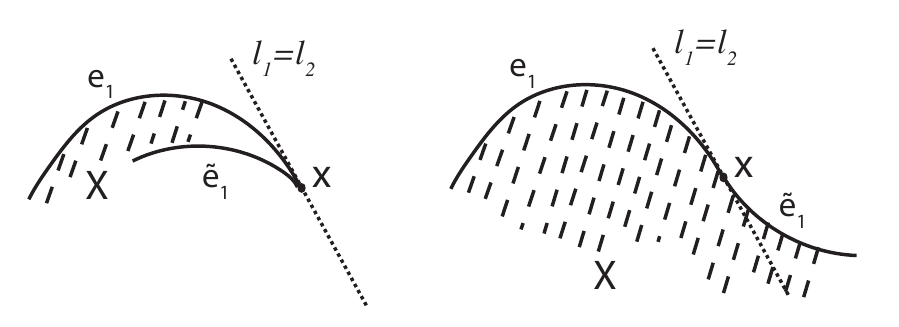} 
\caption{$l_1 = l_2$ and $\tilde{e}_1$ is convex upward in its own coordinate frame.}
\label{fig:24}
\end{figure}  
 If $\tilde{e}_1$ is above $l_1$, we apply the previous argument to get a contradiction. If $\tilde{e}_1$ is below $l_1$, we need another argument as follows.
 
 Let $A_0$ be a point on $\gamma$ such that $\gamma$ passes through $A_0$ before reaching $x$ and $A_0$ is not on $e_1$. This is possible because we assume that $\gamma$ interacts with $e_1$ infinitely many times before arriving at $x$.  $A_0$ cannot be on $\tilde{e}_1$, otherwise $\gamma$ will lie on $\tilde{e}_1$ from $A_0$ to $x$ by Proposition 2.9. It follows that $A_0$ is in the interior of $X$, and thus $\gamma$ is linear near $A_0$. Say $A_1$ is the first point after $A_0$ where $\gamma$ stops being linear, then $A_1$ has to be on $e_1$ and $A_1 \neq x$. Next $\gamma$ can't stay in $e_1$, for $e_1$ is convex downward. Therefore, $\gamma$ leaves $e_1$ at $A_1$ in a straight line towards another point $A_2$ on $e_1$ and $A_2 \neq x$. Afterward, $\gamma$ leaves $e_1$ again, and this process never stops. As a result, $\gamma$ never reaches $x$. This is a contradiction. (Another way to argue is that $\gamma$ is not piecewise $C^2$, because it has infinitely many line segments. In addition, the following corollary will show that $\gamma$ does not hit $e_1$ at all, which provides an even shorter proof.) 

As a summary, what we've proved so far is that when $l_1$ and $l_2$ do coincide, $\gamma$ does not interact with any 1-cell infinitely many times. This completes subsubcase 2. Therefore $\gamma$ interacts with any 1-cell at most finitely many times. For the 2-cells, since $\gamma$ interacts with the boundary of each 2-cell, which consists of finitely many 1-cells, at most finitely many times, $\gamma$ interacts with each 2-cell at most finitely many times. 
\end{proof} 

The last case in the proof of the above proposition is special, and we summarize it in the following corollary. 
\begin{corollary}
\label{cor:2}
Assume $X$ is a region of type I and $x$ is a vertex on $X$. Let $\gamma$ be a shortest-length curve in $X$ which ends at $x$. Furthermore, suppose one side that is adjacent to $x$ is convex downward and the other side is convex upward in their respective coordinate frames, and their tangent lines at $x$ coincide, then $\gamma$ eventually stays in the convex upward side. 

In particular, if $\gamma$ starts at a point not on the convex downward side, then $\gamma$ moves in a straight line toward the convex upward side and then stays in it thereafter; and $\gamma$ never hits the convex downward side. 
\end{corollary}

\begin{proof}
We use our setup as before. Suppose  $A_0 \neq x$ is a point on $\gamma$ that is not in $\tilde{e}_1$. Then $A_0$ is either in $e_1$ or in the interior of $X$. In both cases, $\gamma$ moves in a straight line towards the next point on $\tilde{e}_1$ or $e_1$. If the next point is in $\tilde{e}_1$, $\gamma$ stays in $\tilde{e}_1$ from then on. If the next point is in $e_1$, say $A_1$, then $A_1 \neq x$, because the line segment $[A_1, x]$ intersects $\tilde{e}_1$ at some point besides $x$ due to the fact that $l_1$ and $l_2$ coincide. Then $\gamma$ leaves $e_1$ at $A_1$ and goes in a straight line towards the next point $A_2$ on $\tilde{e}_1$ or $e_1$. This process stops after finitely many times. Eventually $\gamma$ must hit a point on $\tilde{e}_1$ and stay in $\tilde{e}$ thereafter.

In fact, $\gamma$ cannot hit a $A_1$ that is on $e_1$ at all. This is because the line segments $[A_0, A_1]$ and $[A_1, A_2]$ form an angle less than 90 degrees, so we can reduce the length of $\gamma$ using a line segment between one point on $[A_0, A_1]$ and another point on $[A_1, A_2]$ assuming these two points being close enough to $x$ (see Figure~\ref{fig:25}). This is a contradiction. 

As a conclusion, $\gamma$ never hits $e_1$ unless it starts at a point on $e_1$, and $\gamma$ moves in a straight line toward $\tilde{e}_1$ and stays in it till arriving at $x$. 

\begin{figure}[ht]
\includegraphics[width=9cm]
{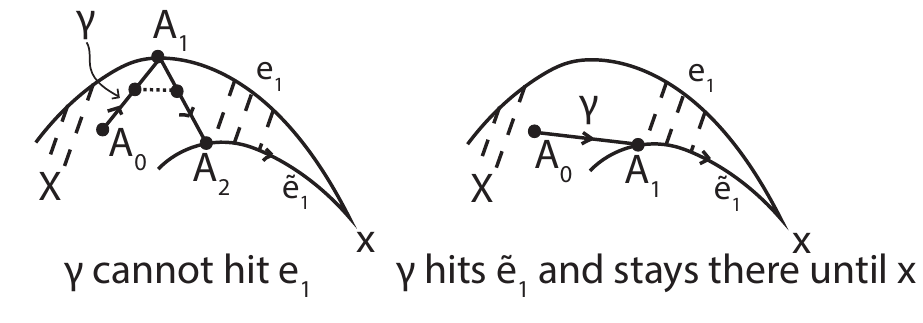} 
\caption{$\gamma$ cannot hit $e_1$ otherwise we can make it shorter.}
\label{fig:25}
\end{figure}  
\end{proof}

\begin{corollary}
\label{cor:3}
Assume $X$ is a region of type I and $\gamma$ is a shortest-length curve between two points in $X$. Then $\gamma$ is an alternating sequence of straight line segments and curves on the boundary of $X$; moreover, each curve on the boundary lies in the convex upward part of a side in its own coordinate frame. In particular, $\gamma$ can be described by finitely many algebraic equations.
\end{corollary}

\begin{proof}
Around every interior point in $X$, $\gamma$ is a straight line. It implies that $\gamma$ stops being a straight line only when it hits a point on the boundary of $X$. As proved in Proposition 2.9, $\gamma$ could stay in the boundary only when it lies on the convex upward part of a side in its own coordinate frame. Therefore $\gamma$ alternates between a line segment and a curve on the boundary. 
\end{proof}

\subsection{Generalize regions of type I}
\label{subsec:3.2}
Let's start with describing a region of type I more formally. First, a region which is below the graph of a polynomial function $f(x)$ can be written as $\{(x, y) | f(x) - y \geq 0 \}$. Next, we rotate it by an angle $\theta$ (see Figure~\ref{fig:37}).
\begin{figure}[ht]
\includegraphics[width=9cm]
{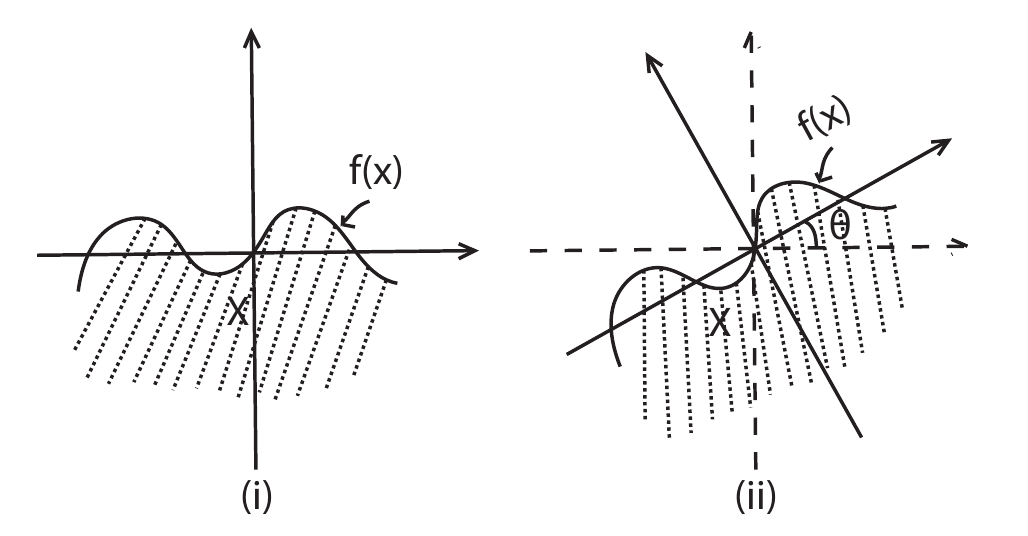} 
\caption{(i) $\{(x, y) | f(x) - y \geq 0 \}$; (ii) $\{ (x, y) | f(\cos(\theta) x + \sin(\theta) y) - (-\sin(\theta) x + \cos(\theta) y) \geq 0 \}.$}
\label{fig:37}
\end{figure} 
Then the new set is:
\begin{equation}
\label{eqn:4}
\{ (x, y) | f(\cos(\theta) x + \sin(\theta) y) - (-\sin(\theta) x + \cos(\theta) y) \geq 0 \}. 
\end{equation}
So a region of type I can be described formally as:
\begin{equation}
\label{eqn:5}
\{f_1 \geq 0 \} \cap \{f_2 \geq 0\} \cap \ldots \cap \{f_k \geq 0 \},
\end{equation}
provided that $\{ f_1 \geq 0\}$ is in the form of (3.1), and the intersection of $\{f_1 \geq 0 \}, \ldots, \{f_k \geq 0 \}$ is simply-connected. For example, the shaded area in Figure~\ref{fig:38} is a region of type I, and it can be described as follows:
\begin{equation*}
\{- (\frac{\sqrt{2}}{2} x + \frac{\sqrt{2}}{2} y )^2 + 2 - (- \frac{\sqrt{2}}{2} x + \frac{\sqrt{2}}{2} y) \geq 0\} \cap \{-x + 1 \geq 0\} \cap \{y + 1 \geq 0\}.
\end{equation*}

\begin{figure}[ht]
\includegraphics[width=5cm]
{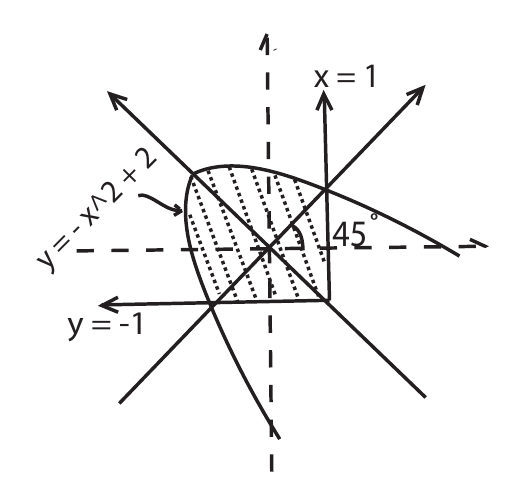} 
\caption{A shaded area bounded by $x=1$, $y=-1$, and the rotated graph of $y=-x^2+2$.}
\label{fig:38}
\end{figure} 

In general, a set in the form of (3.2) might not be connected, and so might not be a region of type I. However, there is a close connection between the two. 

\begin{theorem}
\label{thm:3}
Suppose $X$ is a region that is the intersection of finitely many polynomial half planes determined by the polynomial functions $f_1(x), f_2(x), \ldots, f_k(x)$, i.e.\ , $X$ is in the form of (3.2), then $X$ is a union of finitely many regions of type I satisfying the following properties:  
\begin{enumerate}
\item Every side of a region in $X$ is on the graph of a polynomial function $f_i(x)$, where $1 \leq i \leq k$;
\item No two adjacent sides of a region in $X$ are on the same graph. That is to say, if two adjacent sides are subsets of $\{f_1 \geq 0\}$ and $\{f_2 \geq 0\}$, respectively, then $\{f_1 \geq 0\} \neq \{f_2 \geq 0\}$;
\item Two distinct regions intersect at most at one vertex;
\item No three regions intersect at the same vertex.
\end{enumerate}
\end{theorem}

\begin{proof}
Let $X_1, \ldots, X_k$ be the polynomial half planes corresponding to $f_1(x), \ldots, f_k(x)$, respectively. We prove by induction on $k$. When $k = 1$, $X$ itself is a region of type I. It has only one side which is the graph of $f_1(x)$, therefore property 1 is satisfied. For properties 2, 3, and 4, they are trivially true. When $k \geq 2$, by inductive hypothesis, $X_1 \cap \ldots \cap X_{k-1} = T_1 \cup \ldots \cup T_m$, where $T_j$ ($1 \leq j \leq m$) is a region of type I, and the $T_j$'s satisfy all the four properties.

{\it Part I}: We prove that for each $j$, $T_j \cap X_k$ is a finite union of type I regions satisfying properties 1 - 4. Let the sides of $T_j$ be on the graphs of $g_1, g_2, \ldots, g_r \in \{f_1, f_2, \ldots, f_{k-1}\}$. Then the side on the graph of $g_p, 1 \leq p \leq r$, and the graph of $f_k$ intersect at most at finitely many points. This is because we need to solve the following polynomial equation:
\begin{equation*}
g_p (\cos(\theta) x + \sin(\theta) f_k(x)) = - \sin(\theta) x + \cos(\theta) f_k(x),
\end{equation*}
where $\theta$ is the angle between the coordinate frames of $g_p$ and $f_k$. 
Since we choose $\{f_1 \geq 0\}, \ldots, \{f_k \geq 0\}$ to be distinct sets, the side on the graph of $g_p$ do not overlap partially with the graph of $f_k$. Therefore they have at most finitely many intersection points. It implies that $f_k$ intersects the boundary of $T_j$ at most at finitely many points. Let's denote them as $a_1, a_2, \ldots, a_q$. 

If $q = 0$, the boundary of $T_j$ is either entirely inside $X_k$, or in the complement of $X_k$. When the boundary of $T_j$ is entirely inside $X_k$, $T_j \cap X_k$ is equal to $T_j$, or a region of type I with a disconnected boundary as shown in Figure~\ref{fig:14}. When the boundary of $T_j$ is in the complement of $X_k$, $T_j \cap X_k = \emptyset$ or $X_k$. In all cases, it is true that $T_j \cap X_k$ is a finite union of type I regions satisfying properties 1 - 4 . 

Let $q \geq 1$. Without loss of generality we assume that $a_1, a_2, \ldots, a_q$ are ordered from left to right on the graph of $f_k$ (see Figure~\ref{fig:31}). Let's denote the part on the graph of $f_k(x)$ from $a_i$ to $a_{i+1}$ by $L_i$, where $1 \leq i \leq q-1$. Moreover, denote the part of graph to the left of $a_1$ by $L_0$, and the part to the right of $a_q$ by $L_q$. We note that each $L_i$ without its endpoint(s), lies in either the interior of $T_j$ or the exterior of $T_j$. Let Int $L_i$ be $L_i$ without its endpoint(s) for $0 \leq i \leq q$. Starting from $L_0$, if Int $L_0$ is in the exterior of $T_j$, it does not affect $T_j$ at all. Otherwise, $L_0$ divides $T_j$ into two regions of type I, say $T_j^1$ and $T_j^2$. Both of them has $L_0$ as a side. Next, if Int $L_1$ is in the exterior of $T_j$, discard it. Otherwise, Int $L_1$ is in $T_j^1$ or $T_j^2$, but not both. This is because the two regions have only $L_0$ in common, and $L_0, L_1$ can intersect at most at $a_1$, so $L_1$ can't cross over $L_0$. Suppose Int $L_1$ is inside $T_j^2$, then $L_1$ divides $T_j^2$ into another two regions of type I. Call them $T_j^2$ and $T_j^3$, and they have $L_1$ as a side. We repeat the process for $L_2, L_3, \ldots, L_q$. In the end, $T_j$ is divided by the graph of $f_k$ into finitely many regions of type I with at least one side from the graph of $f_k$ and the rest from the sides of $T_j$ which are on the graphs of $f_1, f_2, \ldots, f_{k-1}$. We select these which are contained in $X_k$, say $T_j^1, T_j^2, \ldots, T_j^t$.  Thus $T_j \cap X_k = T_j^1 \cup T_j^2 \cup \ldots \cup T_j^t$, a finite union of regions of type I.

\begin{figure}[ht]
\includegraphics[width=6cm]
{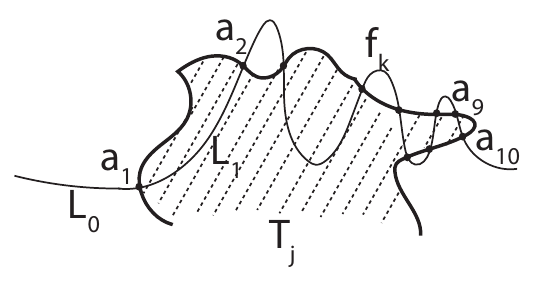} 
\caption{The graph of $f_k(x)$ and the boundary of $T_j$ intersect at $a_1, a_2, \ldots, a_{10}$ ordered from left to right.}
\label{fig:31}
\end{figure}

First, property 1 is satisfied for each of $T_j^1, T_j^2, \ldots, T_j^t$, since the sides are from the graphs of $f_1, f_2, \ldots, f_k$. 

Next, property 2 can be achieved by deleting any vertex that is adjacent to two sides on the same graph. More precisely, assume without loss of generality that $T_j^1$ has a vertex adjacent to two sides on the same graph. This could only happen when the two sides are both on the graph of $f_k$. Furthermore, this vertex cannot be on any other $T_j^2, \dots, T_j^t$. Moreover, this vertex cannot be on any other $T_i^1, T_i^2, \ldots, T_i^{t(i)}$ for $i \neq j$, where $T_i \cap X_k = T_i^1 \cup T_i^2 \cup \ldots \cup T_i^{t(i)}$. It turns out that we can make these two sides into one by removing the vertex in between. 

Then, let's show property 3 is true. Before proving it, we need to first look at the following claim. 

{\it Claim}: Suppose $T_j$ is divided by the graph of $f_k(x)$ into regions $S_1, S_2, \ldots, S_u$ of type I, then $S_{i_1} $ and $S_{i_2}$ share either a side, a vertex, or nothing for $i_1 \neq i_2 \in \{1, 2, \ldots, u\}$.

{\it Proof of claim}: Let the number of $L$'s in $L_0, L_1, \ldots, L_q$ that are contained in $T_j$ be $m$. We prove by induction on $m$. When $m=1$, $T_j$ is divided into two regions and they share a side. When $m \geq 2$, suppose the first $(m-1)$ $L$'s separate $T_j$ into $S_1, \ldots, S_{u-2}, \tilde{S}_{u-1}$, where $\tilde{S}_{u-1} = S_{u-1} \cup S_u$. By inductive hypothesis, given two distinct regions in $S_1, \ldots, S_{u-2}, \tilde{S}_{u-1}$, they share a side, a point, or nothing. Suppose the last $L$ divides $\tilde{S}_{u-1}$ into $S_{u-1}$ and $S_u$, then $S_{u-1}$ and $S_u$ share a side. Furthermore, for $1 \leq i \leq u-2$, if $S_i$ and $\tilde{S}_{u-1}$ share a side, $S_i$ shares the side with either $S_{u-1}$ or $S_u$; say $S_i$ shares the side with $S_{u-1}$, then $S_i$ shares  at most a vertex with $S_u$. If $S_i$ and $\tilde{S}_{u-1}$ share a vertex, then $S_i$ shares at most a vertex with $S_{u-1}$ and $S_u$. Therefore $S_i$ and $S_{u-1}$ have either a common side or a common vertex, or have nothing in common. Similarly, the same is true for $S_i$ and $S_u$, where $1 \leq i \leq u-2$. This completes the proof for the claim.

Now let's prove property 3. By the above claim and the fact that $\{T_j^1, T_j^2, \ldots, T_j^t\} \subseteq \{S_1, S_2, \ldots, S_u\}$, it suffices to prove that $T_j^{i_1} $ and $T_j^{i_2}$ don't have a side in common for $i_1 \neq i_2$. This is true because only the upper or lower part of the side is in $X_k$ if this side comes from the graph of $f_k$. Similarly, only the inner or outer part of the side is in $X_k$ if this side comes from the boundary of $T_j$. So $T_j^{i_1} $ and $T_j^{i_2}$ intersect at most at one vertex when $i_1 \neq i_2$.

Last, property 4 holds. Suppose not, there exist $T_j^{i_1}, T_j^{i_2}, T_j^{i_3}$ such that $i_1 \neq i_2 \neq i_3$ and they have a vertex in common. It implies that this vertex is adjacent to at least six sides. However, this is impossible due to the fact every vertex is either on the boundary of $T_j$ which is adjacent to two sides, or an intersection point in $a_1, a_2, \ldots, a_q$ which is adjacent to four sides. 

{\it Part II}: We are ready to show $X = X_1 \cap \ldots \cap X_{k-1} \cap X_k$ is a finite union of type I regions satisfying all the properties. Since $X_1 \cap \ldots \cap X_{k-1} = T_1 \cup \ldots \cup T_m$, 
\begin{equation*}
X = (T_1 \cap X_k) \cup \ldots \cup (T_m \cap X_k).
\end{equation*}
Properties 1, 2, and 3 are obviously true provided for what we've shown above for $T_j \cap X_k$, $1 \leq j \leq m$. Let us prove property 4. By inductive hypothesis, no three of $T_1, T_2, \ldots, T_m$ intersect at the same vertex. We want to show this is also true for $T_1^1, \ldots, T_1^{t(1)}, \ldots, T_m^1, \ldots, T_m^{t(m)}$, where $T_j^1, \ldots, T_j^{t(j)}$ are type I regions in $T_j \cap X_k$. For the sake of contradiction, suppose this does not hold. It implies that there are $T_{j_1}$ and $T_{j_2}$, $j_1 \neq j_2$, such that they have a vertex $x$ in common; Moreover, one of them, say $T_{j_1}$, has two subregions of type I, say $T_{j_1}^1$ and $T_{j_1}^2$, which have $x$ in common. A picture is shown in Figure~\ref{fig:17}. We'll show this situation does not happen.

\begin{figure}[ht]
\includegraphics[width=6cm]
{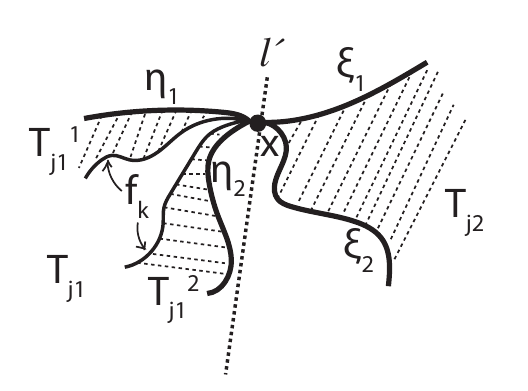} 
\caption{No three regions of type I intersect at the same vertex.}
\label{fig:17}
\end{figure}

{\it Step 1}: Observe that the graph of $f_k$ is contained entirely in $T_{j_1}$ near $x$. This can be seen as follows. Before being cut by the graph of $f_k$, $T_{j_1}$ is a region of type I. If $f_k$ does not go through $x$, we can't have two subregions in $X_k \cap T_{j_1}$ having $x$ as a vertex, so the graph of $f_k$ passes through $x$. Furthermore, it contributes one side to each of $T_{j_1}^1$ and $T_{j_1}^2$. Therefore, the graph of $f_k$ near $x$ is entirely inside $T_{j_1}$ as shown in Figure~\ref{fig:17}. Furthermore, it implies that the graph of $f_k$ does not divide $T_{j_2}$.

{\it Step 2}: Since no three of $T_1, T_2, \ldots, T_m$ intersect at $x$, $x$ is adjacent to four sides: two from $T_{j_1}$ and two from $T_{j_2}$. 
Denote the two sides adjacent to $T_{j_1}$ as $\eta_1, \eta_2$. Then $\eta_1$ and $\eta_2$ cannot be on the same graph by property 2. Similarly, denote the sides adjacent to $x$ in $T_{j_2}$ as $\xi_1, \xi_2$, and $\xi_1, \xi_2$ cannot be on the same graph. Since the graph of a polynomial function can't stop at $x$, we have two possibilities: either $\eta_1, \xi_1$ are on the same graph, and $\eta_2, \xi_2$ are on the same graph; or $\eta_1, \xi_2$ are on the same graph, and $\eta_2, \xi_1$ are on the same graph. Using the configuration in Figure~\ref{fig:17}, we find that only the first possibility makes sense. Call the function, whose graph contains $\eta_1$ and $\xi_1$, $g$; and call the function, whose graph contains $\eta_2$ and $\xi_1$, $h$.

{\it Claim}: The tangent line of $g$ coincides with that of $h$ at $x$.

{\it Proof of claim}: Let the tangent lines of $g$ and $h$ at $x$ be $l_1$, $l_2$, respectively. 

{\it Case 1}: $g$ is on one side of $l_1$ near $x$ as illustrated in Figure 29. For convenience, we assume $g$ is below $l_1$. Since $h$ is below the graph of $g$ in Figure 28, $h$ is also below $l_1$. If $l_2$ is not equal to $l_1$, part of $l_2$ is in the space above $l_1$. If we rotate $l_2$ a little bit, it intersects $h$ at another point besides $x$, because $l_2$ is a tangent line of $h$ at $x$. (A tangent line can be approximated by a sequence of secant lines on a continuously differentiable curve.) It follows that $h$ has another point above $l_1$ near $x$. This is a contradiction. 

\begin{figure}[ht]
\includegraphics[width=6cm]
{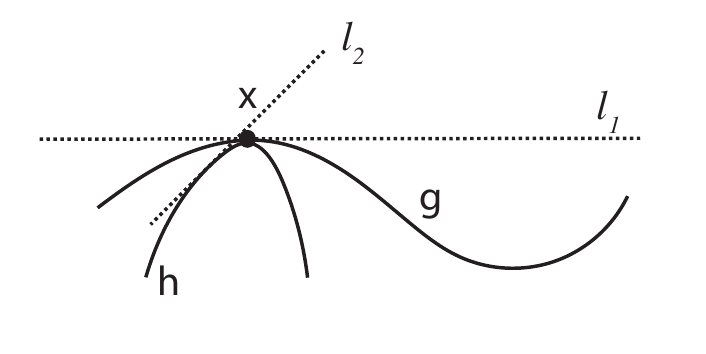} 
\caption{Near $x$, $g$ is on one side of its tangent line at $x$.}
\label{fig:18}
\end{figure}

{\it Case 2}: $g$ is on both sides of $l_1$ near $x$ as illustrated in Figure~\ref{fig:19}.  Since the part of $l_2$ that is above $g$ stays above it even if we rotate $l_2$ a little bit, we can apply the same argument as before. More precisely, if $l_2$ has a positive slope, the right part of $l_2$ increases at a faster rate than $g$, so $l_2$ stays above $g$ for a while to the right. When we rotate $l_2$ a bit, the right part of $l_2$ still stays above $g$. Then we can show there is a point of $h$ above $g$ near $x$, which is a contradiction. The same is true if $l_2$ has a negative slope or $l_2$ is a vertical line. This completes the proof for the claim.
 
\begin{figure}[ht]
\includegraphics[width=5cm]
{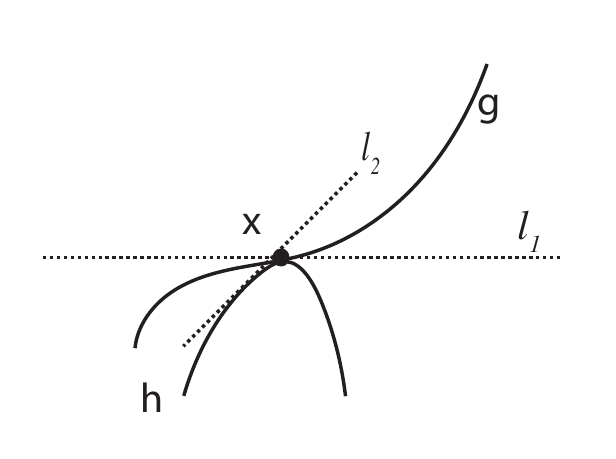} 
\caption{Near $x$, $g$ is on both sides of its tangent line at $x$.}
\label{fig:19}
\end{figure}

{\it Step three}:  Since $f_k$ is below $g$ near $x$, the tangent line of $f_k$ at $x$ is also $l_1$ by the above claim. Draw a perpendicular line $l'$ to $l_1$ through $x$, then $T_{j_1}$ is in the left part of $l'$ and $T_{j_2}$ is in the right part of $l'$ as shown  in Figure~\ref{fig:17}. Furthermore, $l'$ divides $f_k$ into two curves, one to the left of $l'$ and the other to the right. So it is impossible that $f_k$ lies entirely in $T_{j_1}$ near $x$. 
\end{proof}

\begin{theorem}
\label{thm:4}
Suppose $X$ is a region (not necessarily simply-connected)  that is the intersection of finitely many polynomial half planes in the form of (3.2), there exists a cell decomposition of $X$ such that any shortest-length curve between two points in $X$ interacts with each cell at most finitely many times. 
\end{theorem}

\begin{proof}
According to Theorem 3.6, $X$ is a finite union of regions of type I. Denote them as $T_1, T_2, \ldots, T_m$. Furthermore, denote the vertices which connect two of  $T_1, T_2, \ldots, T_m$ as $b_1, b_2, \ldots, b_s$. Let $\gamma$ be a shortest-length curve between two points in $X$. If $\gamma$ is solely in $T_{j_0}$ for some $1 \leq j_0 \leq m$, then we can apply Proposition 3.1. Otherwise, $\gamma$ passes through more than one of $T_1, T_2, \ldots, T_m$.  For any $b_i$, $\gamma$ goes through it at most once. It follows that $\gamma$ interacts with each $T_j$ at most finitely many times, and each intersection is a shortest-length curve in $T_j$ for $1 \leq j \leq m$. Therefore $\gamma$ interacts with each cell of $T_j$ at most finitely many times for $1 \leq j \leq m$. 
\end{proof}

\begin{corollary}
\label{cor:4}
Suppose $X$ is a region that is the intersection of finitely many polynomial half planes in the form of (3.2), then $X = T_1 \cup T_2 \cup \ldots \cup T_m$ is a finite union of regions of type I. Assume $\gamma$ is a shortest-length curve between two points in $X$, then in each $T_j$, $\gamma$ is a disjoint union of finitely many shortest-length curves, each of which is either a point or an alternating sequence of line segments and curves on the boundary of $T_j$, and each curve on the boundary of $T_j$ lies in the convex upward part of a side in $T_j$.  
\end{corollary}

\begin{proof}
It follows from Theorem 3.7 and Corollary 3.5. 
\end{proof}

\section{Real algebraic and semi-algebraic sets in the plane}
\label{sec:4}
\subsection{}
\label{subsec:4.1}
In general, {\bf a semi-algebraic set} in $\mathbb{R}^2$ is a finite union of sets in the following form: 
\begin{equation*}
\{f_1= 0\} \cap \ldots \cap \{f_m = 0\} \cap \{h_1 > 0\} \cap \ldots \cap \{h_p > 0\},
\end{equation*}
where the $f_i$ and the $h_j$ are polynomial functions in $x$ and $y$. 

In \ref{subsec:3.2} we studied sets in the form of $\{f_1 \geq 0 \} \cap \ldots \cap \{f_k \geq 0\}$, which is a basic (closed) semi-algebraic set in $\mathbb{R}^2$. 
Naturally, the next thing to study is the following set:
\begin{equation*}
 X = \{g_1 = 0\} \cap \ldots \cap \{g_m = 0\} \cap \{f_1 > 0 \} \cap \ldots \cap \{f_k > 0\}.
\end{equation*}
Here we use $\{g_i = 0 \}$ as a short-hand notation for 
\begin{equation}
\label{eqn:6}
\{(x, y) | g_i (\cos (\theta_i) x + \sin (\theta_i) y ) - (- \sin(\theta_i) x + \cos (\theta_i)) y = 0\},
\end{equation}
where $\theta_i$ is the angle of rotation of the graph of $g_i(x)$ with respect to the standard Euclidean frame, for $ 1 \leq i \leq m$; and $\{f_j > 0\}$ is a short-hand notation for:
\begin{equation}
\label{eqn:11}
\{(x, y) | f_j (\cos (\beta_j) x + \sin (\beta_j) y ) - (- \sin(\beta_j) x + \cos (\beta_j)) y > 0\},
\end{equation}
where $\beta_j$ is the angle of rotation of the graph of $f_j(x)$, for $ 1 \leq j \leq k$. We ask the same question as before: is there a cell decomposition of $X$ such that every shortest-length curve between two points in $X$ interacts with each cell at most finitely many times? We investigate this question by looking at the following six different cases. 

{\it Case One}: $X = \{g_1 = 0\}$. 
$X$ is a polynomial curve, so we can divide it into intervals using the points $(n, g_1(n))$, $n \in \mathbb{Z}$. Since any shortest-length curve $\gamma$ between two points in $X$ is the curve between them, $\gamma$ interacts with each cell at most once.

{\it Case Two}: $X =  \{g_1 = 0\} \cap \ldots \cap \{g_m = 0\}, m \geq 2$.
In our convention, we assume $\{g_1 = 0\}, \{g_2 = 0\}, \ldots, \{g_m = 0\}$ are distinct sets, therefore $X$ is either empty, or has finitely many points. Thus any shortest-length curve $\gamma$ is a constant path, and so $\gamma$ interacts with each cell at most once.

{\it Case Three}: $X = \{g_1 = 0\} \cap \ldots \cap \{g_m = 0\} \cap \{f_1 > 0 \} \cap \ldots \cap \{f_k > 0\}, m \geq 2, k \geq 1$. 
$X$ is either an empty set or a set of finitely many points. It is the same as case two. 

{\it Case Four}: $X = \{f_1 > 0 \}$. 
$X$ is a polynomial half plane without its boundary, and we call such an $X$ {\bf an open polynomial half plane}. In section~\ref{sec:2}, we've seen a cell decomposition of the closure of $X$. In our case, we need to delete all the cells on the graph of $f_1(x)$ and modify all the 2-cells whose boundaries have edges on the graph of $f_1(x)$. Let $e_2$ be such a 2-cell, then its boundary has an edge lying on the graph of $f_1(x)$. We replace $e_2$ with infinitely many 2-cells in the following way: at a distance of $\frac{1}{2}$ unit below the minimum point of the edge on the graph, we put a 1-cell with the same shape as the edge; then at a distance of $\frac{1}{4}$ unit below the minimum point, we put a 1-cell with the same shape as the edge; in general, we put a 1-cell with the same shape as the edge at a distance of $\frac{1}{2^n}$ unit below the minimum point of the edge for $n \geq 1$ (see Figure~\ref{fig:27}).

\begin{figure}[ht]
\includegraphics[width=5cm]
{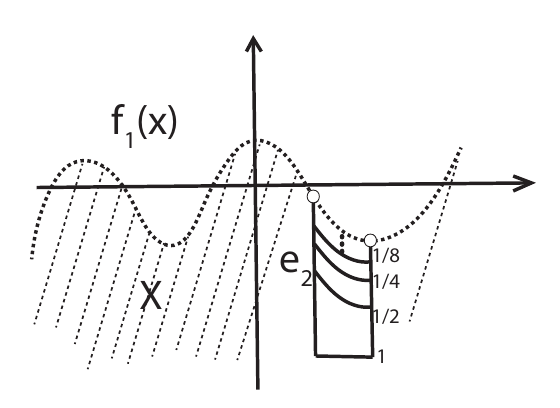} 
\caption{When $X = \{f_1(x) - y > 0 \}$, replace $e_2$ in $\overline{X}$ whose boundary has an edge on the graph of $f_1(x)$ with infinitely many 2-cells.}
\label{fig:27}
\end{figure}

If $\gamma$ is a shortest-length curve between two points in $X$, then $\gamma$ is a straight line segment, because there is no boundary point for $\gamma$ to stop by. It follows that $\gamma$ interacts with each cell at most finitely many times.

{\it Case Five}: $X = \{f_1 > 0 \} \cap \{f_2 > 0 \} \cap \ldots \cap \{f_k > 0 \}, k \geq 2$.
Following the same outline as in section~\ref{sec:3}, we first study a region of type I without its boundary, then we generalize it to the set $\{f_1 > 0 \} \cap \{f_2 > 0 \} \cap \ldots \cap \{f_k > 0 \}$. We definite a region of type I without the boundary as {\bf a region of type II}. 

\begin{proposition}
\label{prop:5}
Assume $X$ is a region of type II, then $X$ has a cell decomposition such that any shortest-length curve between two points in $X$ is a straight line segment that interacts with each cell at most finitely many times.
\end{proposition}

\begin{proof}
First of all, let's construct a cell decomposition for $X$. Since the closure $\overline{X}$ of $X$ has finitely many sides and each side is on the graph of a polynomial function, we call the polynomial functions $f_1(x), f_2(x), \ldots, f_k(x)$. When $k = 1$, we return to case four. When $k \geq 2$, we overlay the cell decompositions of $\{f_1 > 0 \}, \{f_2 > 0 \}, \ldots, \{f_k > 0 \}$ to obtain a cell decomposition of $X$. 

\begin{lemma}
\label{lem:7}
Assume $X$ is a region of type II. Let $f_1(x), f_2(x), \ldots, f_k(x)$ be the polynomial functions whose graphs contain the sides of $\overline{X}$, and let $Y_1, Y_2, \ldots, Y_k$ be the corresponding open polynomial half planes for $f_1(x), f_2(x), \ldots, f_k(x)$, respectively. Then overlaying the cell decompositions of $Y_1, Y_2, \ldots, Y_k$ as shown in case four gives a cell decomposition of $X$. Moreover, the cell decomposition is a countable union of generalized polygons with disjoint interiors; each generalized polygon is enclosed by finitely many edges; and each edge is either a line segment or parallel to a section of one side of $\overline{X}$.  
\end{lemma}

\begin{proof}
It is similar to the proof of Lemma 3.2. When $k = 1$, $X = Y_1$, then each generalized polygon in the cell decomposition of $Y_1$ has four edges with at least two edges being linear; if it has a nonlinear edge, it must be parallel to a section of the graph of $f_1(x)$. When $k \geq 2$, let $X' = Y_1 \cap Y_2 \cap \ldots \cap Y_{k-1}$. By inductive hypothesis, overlaying the cell decompositions of $Y_1, Y_2, \ldots, Y_{k-1}$ provides a cell decomposition of $X'$ in which every edge is either linear or parallel to a section of the graph of $f_1(x), f_2(x), \ldots$, or $f_{k-1}(x)$. Let $e_2$ be one of the 2-cells. 

{\it Case 1}: If $\overline{e}_2$ is completely inside $Y_k$, then $\overline{e}_2$ is at a positive distance from the boundary of $Y_k$, thus $\overline{e}_2$ is covered by finitely many generalized polygons in the cell decomposition of $Y_k$. It follows that after overlaying the cell decompositions of $X'$ and $Y_k$, $\overline{e}_2$ is divided into finitely many generalized polygons whose edges are either line segments or sections parallel to graphs of $f_1(x), f_2(x), \ldots, f_k(x)$. 

{\it Case 2}: If $\overline{e}_2$ is partially inside $Y_k$, then $\overline{e}_2$ is cut off by the graph of $f_k(x)$. As proved in Theorem 3.6, the graph of $f_k(x)$ divides $\overline{e}_2$ into finitely many generalized polygons, each of which has at least one edge on the graph of $f_k(x)$. For each of these generalized polygons, we remove any of its edges that is on the graph of $f_k(x)$ and then check whether it is inside $Y_k$ or not. If it is inside $Y_k$, then after overlaying the cell decompositions of $X'$ and $Y_k$, it is further divided up by the generalized polygons of $Y_k$ into infinitely many generalized polygons whose edges are either line segments or sections parallel to graphs of $f_1(x), f_2(x), \ldots, f_k(x)$ (see Figure~\ref{fig:28}). 

As a summary of the two cases, overlaying the two cell decompositions of $X'$ and $Y_k$ divides $X$ into generalized polygons and thus gives a cell decomposition of $X$. 

\begin{figure}[ht]
\includegraphics[width=7cm]
{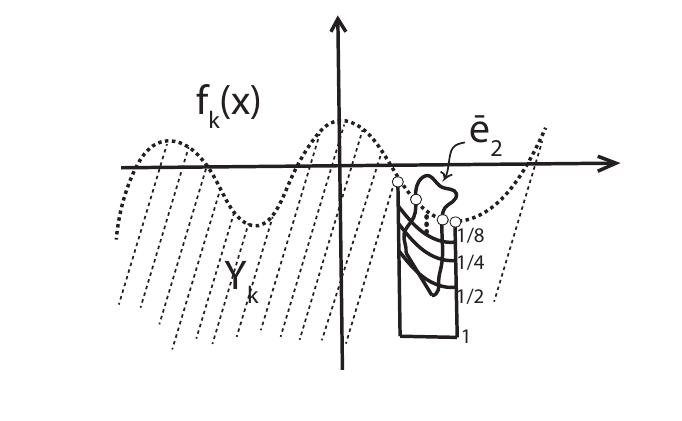} 
\caption{If $\overline{e}_2$ is partially inside $Y_k$, the part inside $Y_k$ is divided into infinitely many generalized polygons.}
\label{fig:28}
\end{figure}
\end{proof}

Let's continue proving the proposition. Suppose $\gamma$ is a shortest-length curve between two points in $X$, then $\gamma$ is a straight line segment in $X$. Given the cell decomposition as in Lemma 4.2, $\gamma$ interacts with each 0-cell at most once; $\gamma$ interacts with each 1-cell at most finitely many times, because every 1-cell is either linear or on the graph of a polynomial function; and so $\gamma$ interacts with each 2-cell at most finitely many times. 
\end{proof}

\begin{theorem}
\label{thm:5}
Suppose $X = \{f_1 > 0 \} \cap \{f_2 > 0 \} \cap \ldots \cap \{f_k > 0 \}$ is the intersection of finitely many open polynomial half planes determined by the polynomial functions $f_1(x)$, $f_2(x)$, $\ldots, f_k(x)$, then $X$ is a disjoint union of finitely many regions of type II. 
\end{theorem}

\begin{proof}
By Theorem 3.6, $\{f_1 \geq 0 \} \cap \{f_2 \geq 0 \} \cap \ldots \cap \{f_k \geq 0 \}$ is a finite union of type I regions whose sides are on the graphs of $f_1(x)$, $f_2(x)$, $\ldots, f_k(x)$, and any two of which intersect at most at one vertex (which is also on a side). It follows that after removing all the sides, $X$ is a disjoint union of finitely many regions of type II. 
\end{proof}

\begin{corollary}
\label{cor:5}
Suppose $X = \{f_1 > 0 \} \cap \{f_2 > 0 \} \cap \ldots \cap \{f_k > 0 \}$ is the intersection of finitely many open polynomial half planes determined by the polynomial functions $f_1(x)$, $f_2(x)$, $\ldots, f_k(x)$, then $X$ has a cell decomposition such that any shortest-length curve interacts with each at most finitely many times.
\end{corollary}

\begin{proof}
By Theorem 4.3, $X = T_1 \cup T_2 \cup \ldots \cup T_m$ is finite disjoint union of regions of type II. Assume $\gamma$ is a shortest-length curve between two points in $X$, then $\gamma$ is a straight line segment contained in $T_j$ for some $1 \leq j \leq m$. It follows from Proposition 4.1 that $T_j$ has a cell decomposition such that $\gamma$ interacts with each cell at most finitely many times. 
\end{proof}

{\it Case Six}: $X = \{g_1 = 0\} \cap \{f_1 > 0 \} \cap \ldots \cap \{f_k > 0\}, k \geq 1$. 
Since $\{g_1 = 0\}$ is a polynomial curve, and  $\{f_1 > 0 \} \cap \ldots \cap \{f_k > 0\}$ is a finite disjoint union of  type II regions, $X$ is a finite disjoint union of curves, each of which is one of the three types (see Figure~\ref{fig:26}): 
\begin{enumerate}
\item the whole polynomial curve itself; 
\item an open half of the polynomial curve; 
\item an open bounded interval of the polynomial curve. 
\end{enumerate}
The reason is because the intersection of $\{g_1 = 0\}$ and a region of type II is an open, connected subset of $\{g_1 = 0\}$. If a curve is one of the three types, we say it is {\bf an open polynomial curve}. 

\begin{figure}[ht]
\includegraphics[width=12cm]
{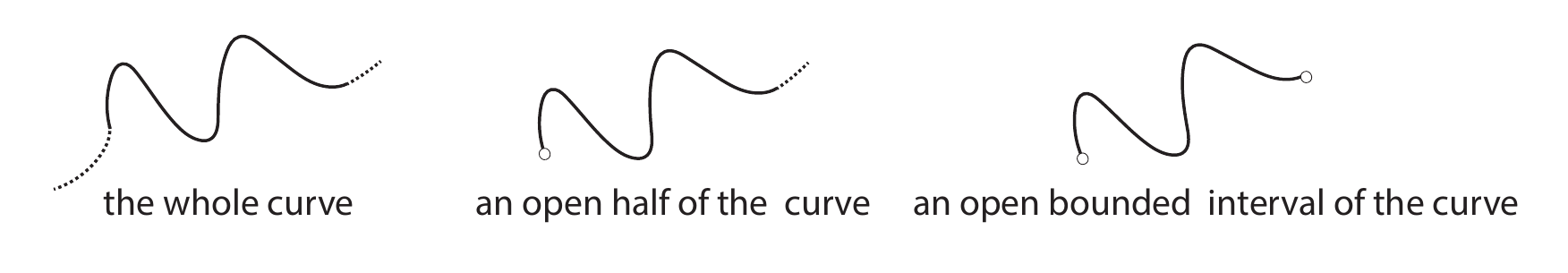} 
\caption{Three kinds of curves in each component of $X = \{g_1 = 0\} \cap \{f_1 > 0 \} \cap \ldots \cap \{f_k > 0\}, k \geq 1$}
\label{fig:26}
\end{figure}

When $X$ is equal to the whole polynomial curve of $g_1(x)$, it is the same as case one. Now suppose $X$ is an open half of the polynomial curve. First, we look at a cell decomposition of $(0, \infty)$ on the real line: the 0-cells are positive integers and $\frac{1}{2^n}$, where $n \in \mathbb{N}$; and the 1-cells are the intervals $(m, m+1)$ and $(\frac{1}{2^n}, \frac{1}{2^{n-1}})$, where $m, n \in \mathbb{N}$. Next, we apply this idea to construct a cell decomposition of $X$. In the coordinate frame of $g_1(x)$, without loss of generality we assume $X$ is the right half of the graph of $g_1(x)$, that is to say, $X = \{(x, g_1(x)) | x > x_0\}$ for some real number $x_0$. Let the 0-cells be $(x_0 + m, g_1(x_0 + m))$ and $(x_0 + \frac{1}{2^n}, g_1(x_0 + \frac{1}{2^n}))$, where $m, n \in \mathbb{N}$. Let the 1-cells be the open intervals between the 0-cells. Then we get a cell decomposition for $X$. Last, any shortest-length curve $\gamma$ between two points in $X$ interacts with each cell at most once. Similarly, when $X$ is an open bounded interval of the polynomial curve, we may assume that $X = \{(x, g_1(x)) | x_0 < x < x_1\}$ for some real numbers $x_0 < x_1$. Let the 0-cells be the integers between $x_0$ and $x_1$, and $x_0 + \frac{1}{2^n}, x_1 - \frac{1}{2^n}$ for $n$ sufficiently large, and let the 1-cells be the open intervals between the 0-cells, thus completing case six.

Let's summarize the six cases in the following theorem.
\begin{theorem}
\label{thm:6}
When $X = \{g_1 = 0\} \cap \ldots \cap \{g_m = 0\} \cap \{f_1 > 0 \} \cap \ldots \cap \{f_k > 0\}$, where $\{g_i = 0\}$ is in the form of (4.1) for $1 \leq i \leq m$, and $\{f_j > 0\}$ is in the form of (4.2) for $1 \leq j \leq k$, then $X$ is either an empty set, or a set of finitely many points, or an open polynomial curve, or a disjoint union of finitely many regions of type II. Moreover, $X$ has a cell decomposition such that any shortest-length curve between two points in $X$ interacts with each cell at most finitely many times.
\end{theorem}

\subsection{}
\label{subsec:4.2}
More generally, let's study a finite union of sets of the following form: 
\begin{equation*}
\{g_1 = 0\} \cap \ldots \cap \{g_m = 0\} \cap \{f_1 > 0 \} \cap \ldots \cap \{f_k > 0\}, 
\end{equation*}
where $\{g_i = 0\}$ is in the form of (4.1) for $1 \leq i \leq m$, and $\{f_j > 0\}$ is in the form of (4.2) for $1 \leq j \leq k$.  Previously we saw that there were three basic building blocks in the union: a point, an open polynomial curve, and a region of type II. How do we put them together in the union? First, we group things of the same kind together and see what their union looks like. Then, we mix combine and check what their union is. We need to consider seven cases.

{\it Case One}: Finitely many points.
When we group all the sets of finitely many points in the union together, we get only finitely many points. The cell decomposition of the union requires only finitely many 0-cells.

{\it Case Two}: Finitely many open polynomial curves.
When we group all the open polynomial curves in the union together, we combine their individual cell decompositions and then add all the intersection points and some endpoints as 0-cells if necessary. More precisely, given two open polynomial curves, they either do not intersect, intersect at finitely many points, or overlap partially.  

\begin{enumerate}
\item When they do not intersect, either the union is disconnected or connected. If the union is disconnected, we use the cell decomposition for each curve to get one for the union. If the union is connected, it must be that the open endpoint of one curve lies on the other curve. In this case, we add a 0-cell for the open endpoint to the cell decomposition of the curve where it lies on (see Figure~
\ref{fig:29} (i)). This may require more 1-cells for the curve if needed. Then we combine the cell decompositions of the two curves together. 
\item When they intersect at finitely many points, we need to include the intersection points as 0-cells to the cell decomposition for each curve and add more 1-cells if needed (see Figure~\ref{fig:29} (ii)). 
\item When they not only intersect at finitely many points, but also have open endpoints lying on other curves, we combine 1 and 2. 
\item When two open polynomial curves overlap partially, their union is again an open polynomial curve for which we know how to find a cell decomposition (see Figure~\ref{fig:29} (iii)).  
\end{enumerate}

\begin{figure}
\includegraphics[width=12cm]
{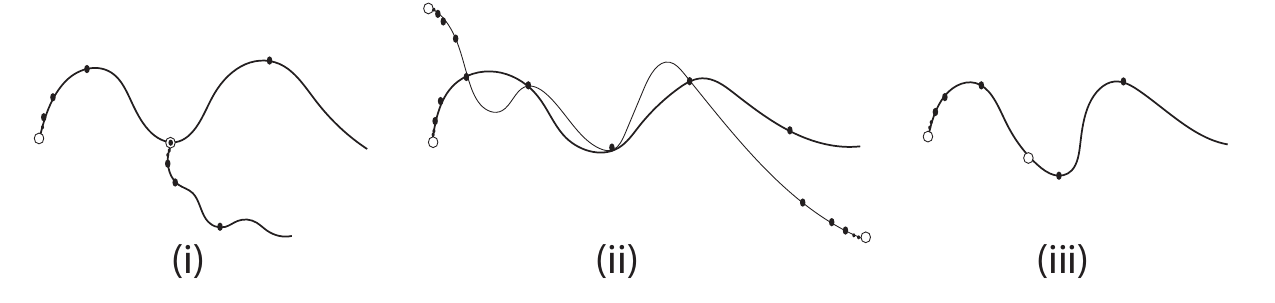} 
\caption{Three examples for a union of two open polynomial curves with their corresponding cell decompositions.}
\label{fig:29}
\end{figure}

Inductively, for a union of finitely many open polynomial curves, we can come up with a cell decomposition. Furthermore, given a shortest-length curve $\gamma$ in the union, it interacts with each 1-cell at most twice. This is because $\gamma$ can enter (or leave) a 1-cell through only its endpoints, otherwise it has to enter (or leave) the 1-cell in the middle, which only occurs when there is another curve intersecting with it or having an open endpoint lying on it. This is impossible in our cell decomposition, since all intersection points and the open endpoints lying on other curves are 0-cells.

{\it Case Three}: Finitely many regions of type II.
When we group all regions of type II in the union together, we don't simply get a disjoint union of regions of type II and we don't even have a disjoint union of simply-connected regions (see Figure~\ref{fig:30}). What does their union look like? 

\begin{figure}[ht]
\includegraphics[width=12cm]
{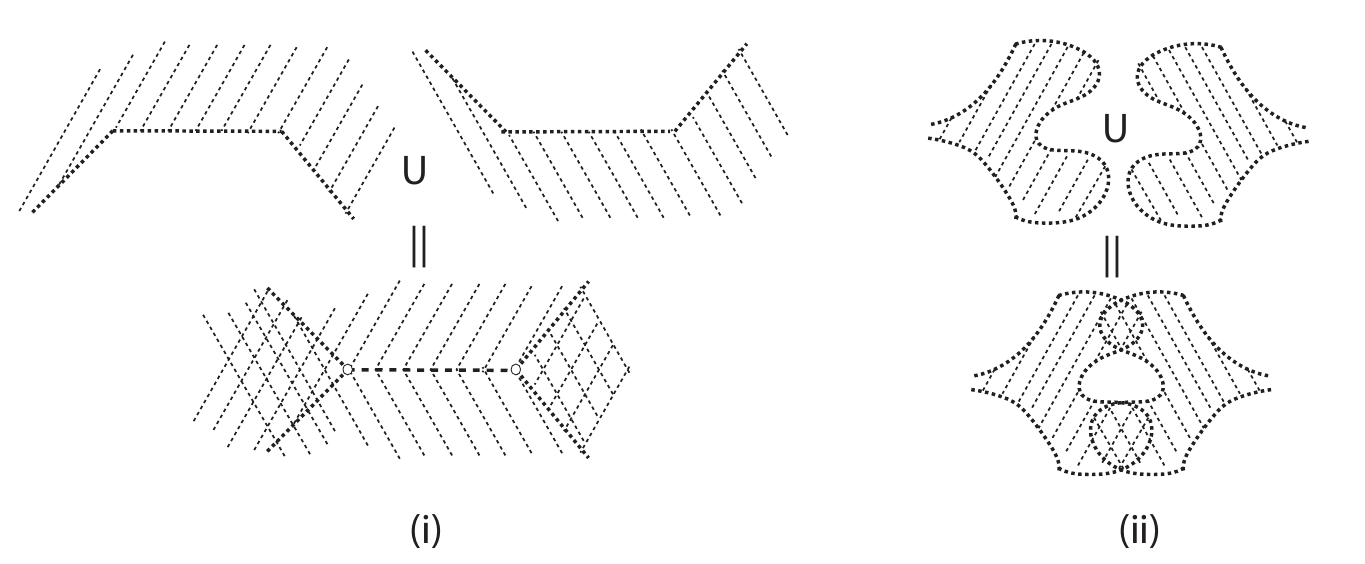} 
\caption{A union of two regions of type II is not necessarily simply-connected: (i) a slit (ii) a hole.}
\label{fig:30}
\end{figure}

First of all, we study a finite union of the boundaries of regions of type II. 
\begin{lemma}
\label{lem:8}
Suppose $X_1, X_2, \ldots, X_n$ are regions of type II with their corresponding boundaries denoted as $\partial X_1, \partial X_2, \ldots, \partial X_n$, then $\partial X_1 \cup \partial X_2 \cup \ldots \cup \partial X_n$ divides the plane into finitely many regions of type I.
\end{lemma}

\begin{proof}
We prove by induction on $n$. Suppose $n = 1$, if $\partial X_1$ is connected, it divides the plane into two regions of type I, each of which has $\partial X_1$ as its side; if $\partial X_1$ is disconnected, say it has $l$ components, then by induction, the plane is divided into $l+1$ regions of type I. 

Suppose $n \geq 2$, let $X' = \partial X_1 \cup \partial X_2 \cup \ldots \cup \partial X_{n-1}$. By inductive hypothesis, $X'$ divides the plane into finitely many regions of type I, say $T_1, T_2, \ldots, T_m$. Suppose $\partial X_n$ consists of $k$ components, namely $C_1, C_2, \ldots, C_k$. Then we induct on $k$. When $k = 1$, since one side in $\partial T_j$ and one side in $C_1$ either overlap partially, or has finitely many intersection points, or do not intersect at all, the intersection of $\partial T_j$ ($1 \leq j \leq m$) with $C_1$ has four possibilities as follows:

\begin{enumerate}
\item $\partial T_j \cap C_1 = \emptyset$. If $C_1$ is in the complement of $T_j$, $T_j$ is not affected at all; if $C_1$ is inside $T_j$, $T_j$ is divided into two regions of type I.

\item $\partial T_j \cap C_1 = $finitely many points. Let the intersection points be $a_1, a_2, \ldots, a_q$. If $C_1$ is not closed, we assume that $a_1, a_2, \ldots, a_q$ are ordered from left to right (see Figure 27). Then we can follow the same proof as in Theorem 3.6. If $C_1$ is closed, we can order $a_1, a_2, \ldots, a_q$ in the clockwise direction and name the side from $a_1$ to $a_q$ as $L_q$. Again we can use the proof in Theorem 3.6. Therefore $T_j$ is divided into finitely many regions of type I.

\item $\partial T_j \cap C_1 = $finitely many closed overlapping intervals. Let the intervals be $O_1$, $O_2$, $\ldots, O_s$, then we can order them either from left to right if $C_1$ is not closed, or in the clockwise direction if $C_1$ is closed (see Figure~\ref{fig:32}). We denote the part of $C_1$ between $O_i$ and $O_{i+1}$ as $L_i$ for $1 \leq i \leq q-1$. If $C_1$ is not closed, let the part to the left of $O_1$ and the part to the right of $O_q$ be $L_0$ and $L_q$, respectively; if $C_1$ is closed, let the part between $O_1$ and $O_q$ be $L_q$. Then the $L_i$ inside $T_j$ divide $T_j$ into finitely many regions of type I.

\begin{figure}[ht]
\includegraphics[width=9cm]
{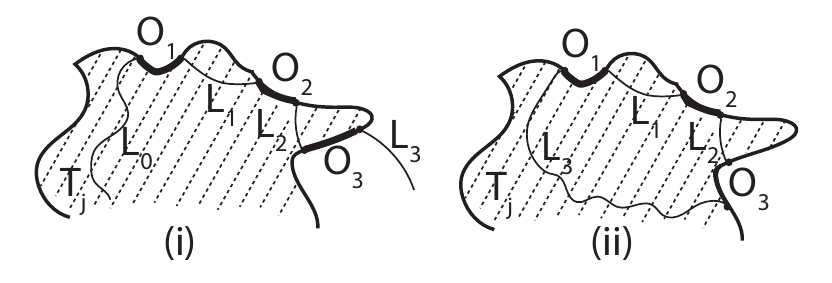} 
\caption{$\partial T_j$ and $C_1$ overlap at the closed intervals $O_1, O_2, O_3$, which are ordered (i) from left to right (ii) in the clockwise direction.}
\label{fig:32}
\end{figure}

\item $\partial T_j \cap C_1 = $finitely many points and finitely many closed overlapping intervals. Suppose the intersection points are $a_1, a_2, \ldots, a_q$, and the intervals are $O_1$, $O_2$, $\ldots, O_s$. We can order all of them on $C_1$ in order (see Figure~\ref{fig:33}). Like before the parts between them inside $T_j$ divide $T_j$ into finitely many regions of type I.

\begin{figure}
\includegraphics[width=10cm]
{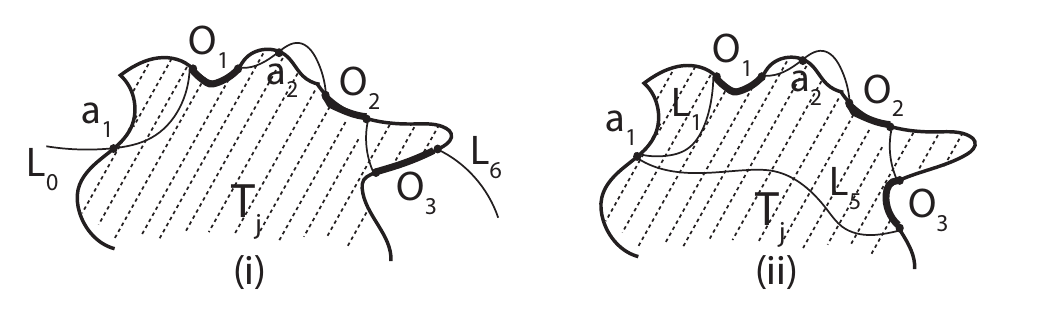} 
\caption{$\partial T_j$ and $C_1$ intersect at $a_1, O_1, a_2, O_2, O_3$, which are in order: (i) from left to right (ii) in the clockwise direction.}
\label{fig:33}
\end{figure}
\end{enumerate}

When $k \geq 2$, suppose $X'$ and $C_1, C_2, \ldots, C_{k-1}$ have divided the plane into finitely many regions of type I. Then $C_k$ divides each region further in the same way as we did for $C_1$. It follows that $X' \cup \partial X_n = \partial X_1 \cup \partial X_2 \cup \ldots \cup \partial X_n$ divides the plane into finitely many regions of type I. 
\end{proof}

Now we are ready to study $X_1 \cup  X_2 \cup \ldots \cup X_n$. By Lemma 4.6, the plane is partitioned into regions of type II, sides without vertices (or {\bf open sides}), and vertices. We select these that are in $X$. We observe that if $L$ is an open side, then $L$ is either in the union or on the boundary. 
\begin{lemma}
\label{lem:9}
Given $L$ as above, then $L$ is either entirely in the union or entirely on the boundary of the union.
\end{lemma}

\begin{proof}
Let $X = X_1 \cup  X_2 \cup \ldots \cup X_n$. Given a point $x$ in $L$, $L$ is in $\partial X_j$ for some $1 \leq j \leq n$, then for any $r > 0$, the open ball $B(x, r)$ must contain a point in $X_j$, thus in $X$. Therefore $L$ is inside the closure of $X$. 

Suppose $L$ has a nonempty intersection with the boundary $\partial X$ of $X$. Since $\partial X$ is contained in $\partial X_1 \cup \partial X_2 \cup \ldots \cup \partial X_n$, $L$ intersects with $\partial X_i$ for some $1 \leq i \leq n$. Let $C$ be a component of $\partial X_i$ that has a nonempty intersection with $L$, then $L \cap C$ is either finitely many points, or finitely many overlapping intervals which are open, half-open, or closed, or both. It follows that  $L \cap C$ can only be $L$, because we've already included all intersection points and the endpoints of all overlapping intervals as vertices in our partition of the plane (see Lemma 4.6). Therefore $L$ is entirely inside $\partial X$. 
\end{proof}

Now we can describe $X$ as follows.
\begin{proposition}
\label{prop:6}
Assume $X_1, X_2, \ldots, X_n$ are regions of type II, then $X_1 \cup X_2 \cup \ldots \cup X_n$ is a disjoint union of finitely many open, connected sets in $\mathbb{R}^2$, such that the boundary of each set, if nonempty, consists of finitely many components, each of which is a finite union of sides belonging to $\partial X_1, \partial X_2, \ldots, \partial X_n$.  
\end{proposition}

\begin{proof}
$X_1 \cup X_2 \cup \ldots \cup X_n$ has finitely many components, because there are only finitely many sides after $\partial X_1 \cup \partial X_2 \cup \ldots \cup \partial X_n$ divides the plane into finitely many regions of type I. 
\end{proof}

{\it Remark}: The proposition seems very simple, and we ask the question: Is this all we can do to characterize $X_1 \cup X_2 \cup \ldots \cup X_n$ (I don't know)?

Since $X$ is a disjoint union of vertices, open sides, and regions of type II, we can gather the cell decomposition for each to obtain a cell decomposition for $X$. Moreover, for any shortest-length curve $\gamma$ in $X$, it interacts with each cell at most finitely many times, because $\gamma$ is locally linear and $X$ is an open set. 

Now we are ready to mix points, open polynomial curves, and regions of type II in the union.

{\it Case Four}: Finitely many points $\cup$ finitely many open polynomial curves. Let $x_1$, $x_2, \ldots, x_l$ be finitely many points, and let $A$ be a finite union of open polynomial curves. In case two, we've studied $A$. If $x_i$ is in $A$, we do nothing, otherwise we add a 0-cell for it, where $1 \leq i \leq l$. 

{\it Case Five}: Finitely many points $\cup$ finitely many regions of type II. Similar as above.

{\it Case Six}: Finitely many open polynomial curves $\cup$ finitely many regions of type II. Let $X = X_1 \cup \ldots \cup X_n$ be a finite union of regions of type II, and let $l_1, l_2, \ldots, l_p$ be open polynomial curves. Since $\partial X \cap l_i$ ($1 \leq i \leq p$) is either finitely many points, or finitely many closed or half-open intervals, or both, $l_i \setminus X$ is a union of finitely many polynomial curves, each of which may or may not lie on the boundary of $X$, and is either closed, or half-open. One example is illustrated in Figure~\ref{fig:34}. 

First, we don't count the endpoints for each polynomial curve in $l_i \setminus X$ ($1 \leq i \leq p$), so we obtain a finite disjoint union of open polynomial curves. Next, we collect them for all $l_i \setminus X$ ($1 \leq i \leq p$), and use the cell decomposition as shown in case two. Then, we add the finitely many endpoints, which were not counted earlier, back to the union as 0-cells. This may require more 1-cells if needed. Last, we include the cell decomposition of $X$ as shown in case three. In this way, we obtain a cell decomposition for $X_1 \cup \ldots \cup X_n \cup l_1 \cup \ldots \cup l_p$. 

\begin{figure}[ht]
\includegraphics[width=6cm]
{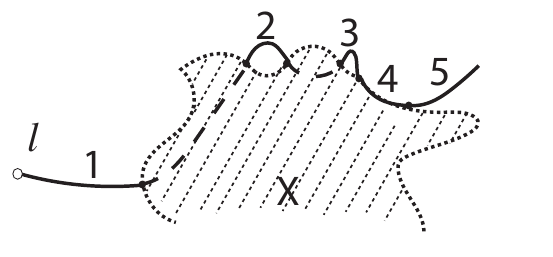} 
\caption{$l \setminus X$ is a union of five polynomial curves, in which the 1st one is half-open, the 2nd one is closed, and the 4th one is on the boundary of $X$.}
\label{fig:34}
\end{figure}

\begin{lemma}
\label{lem:10}
Given the cell decomposition for $X_1 \cup \ldots \cup X_n \cup l_1 \cup \ldots \cup l_p$ as above, and let $\gamma$ be a shortest-length curve between two points in the union, then it interacts with each cell at most finitely many times.
\end{lemma}

\begin{proof}
Let $X_1 \cup \ldots \cup X_n \cup l_1 \cup \ldots \cup l_p$ be denoted as $\tilde{X}$. $\gamma$ interacts with each cell that is in $X$ at most finitely many times, because $\gamma$ is locally linear in $X$. $\gamma$ interacts with each 1-cell that is not in the closure $\overline{X}$ of $X$ at most twice, because $\gamma$ can enter it only through one of its endpoints. Now let $e_1$ be a 1-cell with a nonempty intersection with $\partial X$. Then $e_1$ is in $l_i $ for some $1 \leq i \leq p$. So $e_1 \cap \partial X$ is finitely many open, half-open, or closed overlapping intervals. There can't be any intersection point, because they are already included as 0-cells in the cell decomposition. 

Given one of these intervals, we divide it up using the sides on $\partial X$, then we divide it up again according to the convexity of each side. Let one of them be denoted as $\omega$. Then $\omega$ is on a side of $\partial X_j$ for some $1 \leq j \leq n$. Without loss of generality we assume $\omega \subseteq \partial X_1$. Moreover, $\omega$ is either convex downward, or convex upward and increasing, or convex upward and decreasing. Furthermore, $\omega$ can be either open, or closed, or half-closed. It suffices to show that $\gamma$ interacts with $\omega$ at most finitely many times.

Suppose for the sake of contradiction that there exists a point $x \in \overline{\omega}$ such that $\gamma$ intersects with $\omega$ at infinitely many distinct points converging to $x$, where each point is picked from one interaction between $\gamma$ and $\omega$. We've seen this setup in Proposition 3.1 before. Since $\gamma$ cannot exit the closure $\overline{X}$ of $X$ through any point on $\omega$, $\gamma$ stays inside $\overline{X}$ starting shortly before it reaches $x$. In addition, since $\omega \subseteq \partial X_1$, $\gamma$ is either inside $\overline{X}_1$, or  $X_1^c$ starting shortly before reaching $x$. This is due to the fact that after $\partial X_1 \cup \ldots \cup \partial X_n$ divides the plane, every side is adjacent to two regions of type II, and $X$ could include either of the two regions (see Lemma 4.6). Without loss of generality let us assume that $\gamma$ is inside $\overline{X}_1$ beginning shortly before arriving at $x$.

{\it Case 1}: $\omega$ is convex downward. If $x$ is not a boundary point of $\omega$, we can use the argument as shown in Figure 19 to get a contradiction; if $x$ is a boundary point of $\omega$, then we can apply Lemma 3.3 and Corollary 3.4 to obtain a contradiction. One thing to notice is that $\tilde{e}_1$ in Corollary 3.4 might not be in $\tilde{X}$, in which case $\gamma$ might not be able to reach $x$, which is a contradiction.

{\it Case 2}: $\omega$ is convex upward and increasing (or convex upward and decreasing). By Proposition 2.9, $\gamma$ starts staying in $\omega$ shortly before reaching $x$. Thus we get a contradiction. This completes the proof for the lemma.
\end{proof}

{\it Case Seven}: Finitely many points $\cup$ finitely many open polynomial curves $\cup$ finitely many regions of type II. We combine case two and case six. 

Let's summarize what we have proved in the following theorem.
\begin{theorem}
\label{thm:7}
Suppose $X$ is a finite union of sets in the following form:
\begin{equation*}
\{g_1 = 0\} \cap \ldots \cap \{g_m = 0\} \cap \{f_1 > 0\} \cap \ldots \cap \{f_k > 0\},
\end{equation*}
where $\{g_i = 0\}$ is a short-hand notation for the graph of a polynomial function $g_i(x)$ in the form of (4.1) for $1 \leq i \leq m$, and $\{f_j > 0\}$ is a short-hand notation for an open polynomial half plane associated with a polynomial function $f_j(x)$ in the form of (4.2) for $1 \leq j \leq k$. Then there exists a cell decomposition for $X$ such that any shortest-length curve between two points in $X$ interacts with each cell at most finitely many times. 
\end{theorem}

\section{Conclusion}
Following the footsteps of Whitney, Lojasiewicz, and Hironaka, we try to construct a cell decomposition of real algebraic and semi-algebraic sets in $\mathbb{R}^n$ such that any shortest curve interacts each cell at most finitely many times. This paper starts with two special cases in $n=2$: a real algebraic set, which is the graph of a polynomial function, and a closed semi-algebraic set, which is below the graph of a polynomial function. It follows that there exists a cell decomposition such that every shortest curve between any two points in the set interacts each cell at most finitely many times. Moreover, it interacts with each 0- or 1-cell at most twice. 

Next, we generalize these two special cases to a region of type I, a simply-connected region whose boundary consists of finitely many graphs of polynomial function. We could show that a closed semi-algebraic that is a region of type I admits a cell decomposition with our desired analytical property. More generally, given the intersection $X$ of finitely many polynomial half planes (that is, a half plane whose boundary is the graph of a polynomial function), we show that $X$ is a finite union of regions of type I, say $T_1, \ldots, T_m$. Then given any shortest-length curve $\gamma$ between two points in $X$, it follows that in each $T_j$ the curve $\gamma$ is a finite disjoint union of shortest-length curves, each of which is either a point or an alternating sequence of line segments and curves on the boundary of $T_j$, where each curve on the boundary of $T_j$ lies in the convex upward part of a graph in $T_j$. 

Then, we generalize again to non-closed semi-algebraic sets as follows:
\begin{equation} 
X = \{g_1=0\}\cap\ldots\cap\{g_m=0\}\cap\{f_1 >0\} \cap \ldots \cap \{f_k>0\}, \nonumber
\end{equation}
where each $\{g_i=0\}$ represents the graph of a polynomial function up to rotation, and each $\{f_j>0\}$ represents an open half-plane whose boundary is the graph of a polynomial function up to rotation. We prove that $X$ is either an empty set, or a set of finitely many points, or an open polynomial curve, or a finite disjoint union of regions of type II (the interior of a region of type I). Furthermore, $X$ admits a cell decomposition satisfying our desired analytical condition. More generally, given a finite union of sets in the above form, we could verify that a cell decomposition also exists. 

In the end, we notice that in our construction of cell decompositions, it is possible obtain a triangulation for these special cases, because the cell decomposition in our construction looks like a countable disjoint union of generalized polygons. A generalized polygon replaces each edge of a regular polygon with either a line segment, or a segment of the graph of a polynomial function. These generalized polygons have disjoint interiors. So a triangulation theorem is possible in this direction.

What's more, we observe that any shortest-length curve in these cases is piecewise algebraic, whose formula is described either by a linear equation (for a straight line segment), or by one of the polynomial functions $f_i$ defining the set (for a segment on the graph). So an algebraic characterization of shortest curves in semi-algebraic sets is also possible in this direction.

In the future, we hope to generalize our result to any semi-algebraic set in  the plane. Furthermore, our methods of using a zigzag tangent curve under a convex upward graph of a polynomial function are now being generalized (in progress) to the three-dimensional case. We hope to generalize to even higher dimensions. 

In connection with triangulation of semi-analytic sets, we find that the technique in the above special cases is also applicable  to the analytic functions instead of polynomial functions, as long as the functions have finitely many inflection points and local minimum points.

\bibliographystyle{amsplain}

\begin{thebibliography}{10}
\bibitem{H} H. Hironaka, \textit{Triangulations of algebraic sets}, Proc. Sympos. Pure Math., vol. 29, Amer. Math. Soc., Providence, R.I., 1975, pp. 165-185.

\bibitem{L} S. Lojasiewicz, \textit{Triangulations of semi-analytic sets}, Ann. Scuola Norm. Sup. Pisa (5) (3) {\bf18} (1964), 449-474.

\bibitem{RefM&R} Sebasti\'an Montiel and Antonio Ros, \textit{Curves and Surfaces}, 2nd ed., 
 Amer. Math. Soc., Real Sociedad Matem\'atica Espa\~nola, 2009, pp. 5.

\bibitem {V} B. L. van der Waerden, Topologische Begründung des Kalküls der abzählenden Geometrie, Math. Ann. 102 (1929), 337-362.

\bibitem{W} H. Whitney, \textit{Elementary structure of real algebraic varieties}, Ann. of Math. (2) {\bf66} (1957), 545-556.
\end{thebibliography}

\end{document}